\documentclass[10pt]{amsart}\oddsidemargin -1mm \evensidemargin -1mm \topmargin -15mm\headheight5mm \headsep 3.5mm \textheight 245mm\textwidth 170mm 
\usepackage{amsfonts, amssymb, amsthm, amsmath, euscript, hhline}
\usepackage{stmaryrd}
\theoremstyle{definition}\newtheorem{theorem}{Theorem}[section]\newtheorem{lemma}[theorem]{Lemma}\newtheorem{remark}[theorem]{Remark}\newtheorem{proposition}[theorem]{Proposition}\newtheorem{corollary}[theorem]{Corollary}\newtheorem{definition}[theorem]{Definition}\newtheorem{Ex}[theorem]{Example}
\DeclareMathOperator{\Ima}{\mathrm{im}}\DeclareMathOperator{\RS}{\mathrm{RS}}\DeclareMathOperator{\wid}{A}\DeclareMathOperator{\nint}{\alpha}\DeclareMathOperator{\APAut}{\mathrm{Aut}}\DeclareMathOperator{\APrk}{\mathrm{rk}}\DeclareMathOperator{\Sa}{\mathbf{S}}\DeclareMathOperator{\Irr}{\mathrm{Irr}}\DeclareMathOperator{\Ue}{\mathrm{U}}\DeclareMathOperator{\seq}{\mathrm{seq}}\DeclareMathOperator{\cls}{\mathrm{cls}}\DeclareMathOperator{\injo}{\mathrm{inj_0}}\DeclareMathOperator{\Ver}{{\it M}}\DeclareMathOperator{\Adj}{\mathrm{Adj}}\DeclareMathOperator{\APinj}{\mathrm{inj}}\DeclareMathOperator{\Ann}{\mathrm{Ann}}\DeclareMathOperator{\ZUe}{\mathrm{ZU}}\DeclareMathOperator{\Var}{\mathrm{Var}}\DeclareMathOperator{\ZVar}{\mathrm{ZVar}}
\newcommand{\APbr}[1]{\mathrm{span}\{#1\}}
\newcommand{\APbre}[1]{\langle #1\rangle}
\newcommand{\apro}[1]{{\rm\setcounter{AP}{#1}\roman{AP}}}

\newcounter{AP}
\begin{document}
\author{Ivan Penkov}\address{Ivan Penkov: Jacobs University Bremen, Campus Ring 1, 28759 Bremen, Germany}\email{i.penkov@jacobs-university.de}\author{Alexey Petukhov}\address{Alexey Petukhov: The University of Manchester, Oxford Rd M13 9PL, Manchester, UK\\ on leave from the Institute for Information Transmission Problems, Bolshoy Karetniy 19-1, Moscow 127994, Russia}
\email{alex--2@yandex.ru}
\title{Annihilators of highest weight $\frak{sl}(\infty)$-modules}\maketitle
\begin{abstract} We give a criterion for the annihilator in $\Ue(\frak{sl}(\infty))$ of a simple highest weight $\frak{sl}(\infty)$-module to be nonzero. As a consequence we show that, in contrast with the case of $\frak{sl}(n)$, the annihilator in $\Ue(\frak{sl}(\infty))$ of any simple highest weight $\frak{sl}(\infty)$-module is integrable, i.e., coincides with the annihilator of an integrable $\frak{sl}(\infty)$-module. Furthermore, we define the class of ideal Borel subalgebras of $\frak{sl}(\infty)$, and prove that any prime integrable ideal in $\Ue(\frak{sl}(\infty))$ is the annihilator of a simple $\frak b^0$-highest weight module, where $\frak b^0$ is any fixed ideal Borel subalgebra of $\frak{sl}(\infty)$. This latter result is an analogue of the celebrated Duflo Theorem for primitive ideals.\medskip


 \end{abstract}
\section{Introduction}

The base field is $\mathbb C$. If $\frak g$ is a semisimple finite-dimensional Lie algebra, the celebrated Duflo Theorem states that any primitive two-sided ideal in the enveloping algebra $\Ue(\frak g)$ of $\frak g$ (i.e., any annihilator of a simple $\Ue(\frak g)$-module) is the annihilator of a simple highest weight $\frak g$-module.

The purpose of the present paper is to study primitive ideals in the enveloping algebra $\Ue(\frak{sl}(\infty))$ of the infinite-dimensional Lie algebra $\frak{sl}(\infty)$, and in particular to obtain a partial analogue of Duflo's Theorem for $\frak{sl}(\infty)$. Recall that the Lie algebra $\frak{sl}(\infty)$ can be defined in several equivalent ways, for instance as a direct limit $\varinjlim\limits_{n\ge 2}\frak{sl}(n)$~\cite{Ba1, Ba2, DP}.

The study of two-sided ideals in $\Ue(\frak{sl}(\infty))$ has been initiated by A.~Zhilinskii~\cite{Zh1, Zh2, Zh3}, and has been continued in~\cite{PP}. Zhilinskii's idea has been to study the joint annihilators of certain systems of $\frak{sl}(n)$-modules for variable $n>2$, more precisely, the joint annihilators of coherent local systems of finite-dimensional $\frak{sl}(n)$-modules as defined in~\cite{Zh1}. Zhilinskii has also provided a classification of coherent local systems~\cite{Zh1, Zh2}. We call the ideals introduced by Zhilinskii {\it integrable} (see Section~\ref{SSzhi} for the precise definition).

A corollary of the results in~\cite{PP} is that the associated ''variety'' of an arbitrary ideal in $\Ue(\frak{sl}(\infty))$ coincides with the associated ``variety'' of some integrable ideal in $\Ue(\frak{sl}(\infty))$.
We do not know whether any ideal in $\Ue(\frak{sl}(\infty))$ is integrable, however in the present paper we prove that the annihilator of any highest weight $\frak{sl}(\infty)$-module is an integrable ideal in $\Ue(\frak{sl}(\infty))$.

In order to recall the definition of a highest weight $\frak{sl}(\infty)$-module, we first need to recall the definition of a splitting Borel subalgebra of $\frak{sl}(\infty)$. According to~\cite{DP}, a splitting Borel subalgebra is a subalgebra of $\frak{sl}(\infty)$ which can be obtained as a direct limit of $\varinjlim\frak b_n$ of Borel subalgebras $\frak b_n\subset\frak{sl}(n)$ for a suitable presentation $\frak{sl}(\infty)$ as a direct limit $\varinjlim\limits_{n\ge 2}\frak{sl}(n)$. In contrast with the finite-dimensional case, the splitting Borel subalgebras of $\frak{sl}(\infty)$ are not conjugate by the group of automorphisms of $\frak{sl}(\infty)$; in fact, there are uncountably many conjugacy classes (and even isomorphism classes) of splitting Borel subalgebras of  $\frak{sl}(\infty)$. However, a $\frak b$-highest weight module is defined as usual as an $\frak{sl}(\infty)$-module generated by a 1-dimensional $\frak b$-submodule.

The difference between the structure of ideals in $\Ue(\frak{sl}(\infty))$ and in $\Ue(\frak g)$ for a finite-dimensional semisimple $\frak g$, becomes apparent in the fact that the annihilators in $\Ue(\frak{sl}(\infty))$ of many simple highest weight modules equal to zero. In this paper we give an explicit criterion for a simple $\frak b$-highest weight module to have nonzero annihilator. A further central result which we establish is that the annihilator of any $\frak b$-highest weight $\frak{sl}(\infty)$-module is integrable.

Our third notable result is an analogue of Duflo's Theorem. We define a special class of splitting Borel subalgebras $\frak b^0\subset\frak{sl}(\infty)$, which we call ideal, and prove that any prime integrable ideal of $\Ue(\frak{sl}(\infty))$ is the annihilator of a simple $\frak b^0$-highest weight module for any $\frak b^0$. The ideal Borel subalgebras $\frak b^0$ have the property that the adjoint representation of $\frak{sl}(\infty)$ is a $\frak b^0$-highest weight module.

The paper is structured as follows. In Section~\ref{Spre} we review some well known and some not so well known results about the Lie algebra $\frak{sl}(\infty)$ and its representations. Section~\ref{Sstate} contains a precise statement of our main results. 
The proofs are given in Sections~\ref{Spr1ta},~\ref{Spr1tb} and~\ref{Spr2t}. In Section~\ref{Snew} we characterize simple $\frak{sl}(\infty)$-modules which are determined up to isomorphism by their annihilators in $\Ue(\frak{sl}(\infty))$, under the assumption that the annihilator is integrable. 

{\bf Acknowledgements.} We thank I. Losev for his lectures on representation theory which helped us understand the combinatorics of two-sided ideals of $\Ue(\frak{sl}(n))$. We also thank the Max-Planck Institute for Mathematics in Bonn for its support and hospitality in 2012-2013. The first author acknowledges partial support by DFG through the Priority Program ''Representation Theory'' (SPP 1388). The second author acknowledges partial support by Jacobs University Bremen. 

\section{Preliminaries}\label{Spre}
\subsection{The Lie algebra $\frak{sl}(\infty)$}\label{SSslnot}The superscript $^*$ indicates dual space, and $\Sa^\cdot(\cdot)$ and $\Lambda^\cdot(\cdot)$ stand respectively for symmetric and exterior algebra. For a Lie algebra $\frak g$, $\Ue(\frak g)$ stands for the
universal enveloping algebra of $\frak g$. If $M$ is a $\frak g$-module, then $\Ann_{\Ue(\frak g)}M$ denotes the annihilator of $M$ in $\Ue(\frak g)$.

The Lie algebra $\frak{gl}(\infty)$ can be defined  as the Lie algebra of matrices $(a_{ij})_{i, j\in\mathbb Z_{>0}}$ each of which has at most finitely many nonzero entries. Equivalently, $\frak{gl}(\infty)$ can be defined by giving an explicit basis. Let $\{e_{ij}\}_{i, j\in\mathbb Z_{>0}}$ be a basis of a countable-dimensional vector space denoted by $\frak{gl}(\infty)$. Set $\widehat{\frak h}:=\APbr{e_{ii}}_{i\in\mathbb Z_{>0}}$. The structure of a Lie algebra on $\frak{gl}(\infty)$ is given by the formula
$$[e_{ij},
e_{kl}]=\delta_{jk}e_{il}-\delta_{il}e_{kj},$$
where $i, j\in\mathbb Z_{>0} $ and $\delta_{mn}$ is Kronecker's delta.

Next, one defines $\frak{sl}(\infty)$ as the commutator subalgebra of $\frak{gl}(\infty)$: $$\frak{sl}(\infty):=[\frak{gl}(\infty), \frak{gl}(\infty)].$$ We set$$\frak h:=\widehat{\frak h}\cap\frak{sl}(\infty).$$ Clearly, $\widehat{\frak h}$ is a maximal commutative subalgebra of $\frak{gl}(\infty)$, and $\frak h$ is a maximal commutative subalgebra of $\frak{sl}(\infty)$. Moreover, $\frak{gl}(\infty)$ has the following root decomposition $$\frak{gl}(\infty)=\widehat{\frak h}\oplus\bigoplus_{\alpha\in\Delta}\frak{gl}(\infty)^{\alpha},$$
similar to the usual root decomposition of $\frak{gl}(n)$. Here $\Delta=\{\varepsilon_i-\varepsilon_j\}_{i, j\in\mathbb Z_{>0}}$ where the system of vectors $\{\varepsilon_j\}_{j\in\mathbb Z_{>0}}$ in $\widehat{\frak h}^*$ is dual to the basis $\{e_{ii}\}_{i\in\mathbb Z_{>0}}$ of $\widehat{\frak h}$. The Lie subalgebra $\frak{sl}(\infty)$ inherits this root decomposition:$$\frak{sl}(\infty)=\frak h\oplus\bigoplus_{\alpha\in\Delta}\frak{sl}(\infty)^{\alpha},$$where $\frak{sl}(\infty)^\alpha=\frak{gl}(\infty)^\alpha$ for $\alpha\in\Delta$.

It is not difficult to prove that any Lie algebra obtained as a direct limit $\varinjlim\limits_{n\ge2}\frak{sl}(n)$ is isomorphic to $\frak{sl}(\infty)$ as defined above. Moreover, a general definition of a {\it splitting Cartan subalgebra} $\frak h'$ of $\frak{sl}(\infty)$ is as a  direct limit of Cartan subalgebras $\frak h_n'$ of $\frak{sl}(n)$, where $\frak{sl}(\infty)$ is identified with $\varinjlim\limits_{n\ge2}\frak{sl}(n)$. Then, as noted in~\cite{DPSn}, all splitting Cartan subalgebras of $\frak{sl}(\infty)$ are conjugate via the automorphism group $\APAut \frak{sl}(\infty)$. This enables us to henceforth restrict ourselves to considering only the fixed splitting Cartan subalgebra $\frak h$ of $\frak{sl}(\infty)$ introduced above.

 A {\it splitting Borel subalgebra} of $\frak{sl}(\infty)\cong \varinjlim\limits_{n\ge2}\frak{sl}(n)$ can be defined as a direct limit $\varinjlim\limits_{n\ge2}\frak b_n$ of Borel subalgebras $\frak b_n\subset\frak{sl}(n)$, see~\cite{DP}. Since a general splitting Borel subalgebra of $\frak{sl}(\infty)$ is conjugate under $\APAut(\frak{sl}(\infty))$ to a splitting Borel subalgebra containing our fixed splitting Cartan subalgebra $\frak h\subset\frak{sl}(\infty)$, in what follows we only consider splitting Borel subalgebras containing $\frak h$.
 The latter Borel subalgebras are given by the following construction. We say that a subset $\Delta^*\subset\Delta$ is a subset {\it of positive roots} if

(1) for any root $\alpha\in\Delta$, precisely one of $\alpha$ and $-\alpha$ belongs to $\Delta^*$;

(2) $\alpha, \beta\in\Delta^*$ and $\alpha+\beta\in\Delta$ imply $\alpha+\beta\in\Delta^*$.\\
To any positive subset of roots $\Delta^*$ we assign the Borel subalgebra $\frak b(\Delta^*):=\frak h\bigoplus\limits_{\alpha\in\Delta^*}\frak{sl}(\infty)^{\alpha}$ of $\frak{sl}(\infty)$, and in this way we obtain all splitting Borel subalgebras of $\frak{sl}(\infty)$ containing $\frak h$.

This leads naturally to the observation~\cite{DP} that the splitting Borel subalgebras containing $\frak h$ are in one-to-one correspondence with linear orders on $\mathbb Z_{>0}$: given such a linear order $\prec$, the corresponding subset of positive roots is $\{\varepsilon_i-\varepsilon_j\}_{i\prec j}$.

It is easy to see that different Borel subalgebras containing $\frak h$ do not have to be $\APAut\frak{sl}(\infty)$-conjugate, as they simply may not be isomorphic as abstract Lie algebras. Consider, for instance, the following three linear orders on $\mathbb Z_{>0}$:

(\apro{1}) $...\prec 5\prec 3\prec 1\prec 2\prec 4\prec 6\prec...$,

(\apro 2) $1\prec 2\prec 3\prec...$,

(\apro 3) $1\prec3\prec5\prec...\prec2n+1\prec...\prec2n\prec...\prec4\prec2$.\\
The reader can check that the corresponding Borel subalgebras are not isomorphic as Lie algebras.

\subsection{$S$-notation}\label{SSsnot}
Let $S$ be a subset of $\mathbb Z_{>0}$. We denote by $\frak{sl}(S)$ the subalgebra of $\frak{sl}(\infty)$ spanned by \begin{center}$\{e_{ij}\}_{i, j\in S, i\ne j}\hspace{10pt}$and\hspace{10pt}$\{e_{ii}-e_{jj}\}_{i, j\in S}.$\end{center}
Then $\frak{sl}(\mathbb Z_{>0})=\frak{sl}(\infty)$.

Set $\frak h_S:=\frak h\cap\frak{sl}(S)$. Note that

(1) if $S$ is finite, then $\frak{sl}(S)$ is isomorphic to $\frak{sl}(n)$ where $n=|S|$ is the cardinality of $S$, and $\frak h_S$ is a Cartan subalgebra of $\frak{sl}(S)$;

(2) if $S$ is infinite, then $\frak{sl}(S)$ is isomorphic to $\frak{sl}(\infty)$, and $\frak h_S$  is a splitting Cartan subalgebra of $\frak{sl}(S)$.\\
Next, we fix a splitting Borel subalgebra $\frak b\supset\frak h$ of $\frak{sl}(\infty)$ and put $\frak b_S:=\frak{sl}(S)\cap\frak b$. We note that

(1) if $S$ is finite, then $\frak b_S$ is a Borel subalgebra of $\frak{sl}(S)$,

(2) if $S$ is infinite, then $\frak b_S$ is a splitting Borel subalgebra of $\frak{sl}(S)$.

Let $\mathbb C^S$ denote the set of functions from $S$ to $\mathbb C$. Clearly, $\mathbb C^S$ is a vector space of
dimension~$|S|$. When $S=\{1,..., n\}$ we write simply $\mathbb C^n$ instead of $\mathbb C^{\{1,..., n\}}$. There is a surjective
homomorphism from $\mathbb C^S$ to $\frak h_S^*$:\begin{equation}f\mapsto \lambda_f,\hspace{10pt}
\lambda_f(e_{ii}-e_{jj})=f(i)-f(j).\end{equation}

For any $f\in \mathbb C^S$ we denote by $|f|$ the the cardinality of the image of~$f$. A weight $\lambda\in\frak h^*_S$ is {\it $\frak{sl}(S)$-integral}, or simply {\it integral}, if $\lambda(e_{ii}-e_{jj})\in\mathbb Z$ for all $i, j\in S$. Respectively, a function $f\in\mathbb C^S$ is {\it integral} if $f(i)-f(j)\in\mathbb Z$ for all $i, j\in S$. A function is {\it almost integral} if there exists a finite subset $F\subset S$ such that $f|_{S\backslash F}$ is integral.

If $\frak b_S\supset\frak h_S$ is a fixed splitting Borel subalgebra of $\frak{sl}(S)$, then an integral weight $\lambda\in\frak h_S^*$ is {\it $\frak b_S$-dominant} if $\lambda(e_{ii}-e_{jj})\ge0$ for $i\prec j$ where the order $\prec$ on $S$ is determined by $\frak b_S$. Respectively, an integral function $f\in \mathbb C^S$ is {\it $\prec$-dominant} if $f(i)-f(j)\ge0$ for $i\prec j$.

Let $\prec$ be a linear order on $S$, and let $S=S_1\sqcup ...\sqcup S_t$ be a finite partition of $S$. We say that the partition $\{S_i\}_{i\le t}$ is {\it compatible} with the order $\prec$ if\begin{center}$i_0\prec j_0\Leftrightarrow i<j$\end{center}for any
$i\ne j\le t$ and any $i_0 \in S_i, j_0\in S_j$. Finally, we say that $f\in \mathbb C^S$ is {\it locally constant with respect to} $\prec$ if there exists
a compatible partition $S_1\sqcup...\sqcup S_t$ of $S$, such
that $f$ is constant on $S_i$ for any $i\le t$.

We call a splitting Borel subalgebra $\frak b_S\supset \frak h_S$ of $\frak{sl}(S)$ {\it ideal} if there is a partition $S=S_1\sqcup S_2\sqcup S_3$, compatible with the order $\prec$ defined by $\frak b_S$, such that

(a) $S_1$ is countable and $\prec$ restricted to $S_1$ is isomorphic to the standard order on $\mathbb Z_{>0}$,

(b) $S_3$ is countable and $\prec$ restricted to $S_3$ is isomorphic to the standard order on $\mathbb Z_{<0}$\\
($S_2$ may be empty). Clearly the Borel subalgebra defined by the above order (\apro{3}) is ideal, while the Borel subalgebras defined by (\apro{1}) and (\apro{2}) are not ideal.

\subsection{Highest weight $\frak{sl}(S)$-modules}\label{SSmi}
Fix a splitting Borel subalgebra $\frak b_S$ of $\frak{sl}(S)$, corresponding to a linear order $\prec$ on $S$. A {\it Verma module} is defined as an induced module
$$\Ver_{\frak b_S}(f):=\Ue(\frak{sl}(S))\otimes_{\Ue(\frak b_S)}\mathbb C_f,$$
where $\mathbb C_f$ is a one-dimensional $\frak b_S$-module
determined by a weight $\lambda_f\in\frak h_S^*$. By definition, a $\frak b_S$-{\it{}highest weight module} is an $\frak{sl}(\infty)$-module isomorphic to a quotient of $\Ver_{\frak b_S}(f)$. It is not difficult to prove that $\Ver_{\frak b_S}(f)$ has a unique simple quotient $L_{\frak b_S}(f)$, see~\cite{DP}. 

As $S$ and $\frak b_S$ are fixed, in the rest of Section~\ref{SSmi} we write simply $M(f)$ and $L(f)$ instead of $M_{\frak b_S}(f)$ and $L_{\frak b_S}(f)$. We fix also a function $f\in\mathbb C^{\mathbb Z_{>0}}$ and a highest weight vector $v$ of $L(f)$. For any subset $S'\subset S$ we denote by $\widehat L(f|_{S'})$ the $\frak{sl}(S')$-submodule of $L(f)$ generated by $v$. Obviously $\widehat L(f|_{S'})$ is a quotient of $M(f|_{S'})$, and $L(f|_{S'})$ is the unique simple quotient of $\widehat L(f|_{S'}).$

For any finite subset $F\subset S$, let $w_F$ be a fixed highest weight vector in $M(f|_F)$, and let $v_F$ be its image in $L(f|_F)$. Let $F\subset F'\subset S$ be two finite subsets. Then there exists a unique morphism of $\frak{sl}(F)$-modules $$\psi_{F, F'}: M(f|_F)\to L(f|_{F'})$$
such that $w_F\mapsto v_{F'}$. It is clear that if $F''\supset F'\supset F$ then $\ker \psi_{F, F''}\subset \ker \psi_{F, F'}$. Since the $\frak{sl}(F)$-module $M(f|_F)$ has finite length, there exists a sufficiently large finite set $\bar F\supset F$ such that $\ker \psi_{F, \bar F}\subset \ker\psi_{F', F}$ for any finite set $F'\supset F$. We put $\psi_F:=\psi_{F, \bar F}$.

\begin{proposition}\label{Pfinf}The $\frak{sl}(F)$-module $\widehat L(f|_F)$ is isomorphic to the image of $\psi_F$.\end{proposition}
\begin{proof}Let $F\subset F'$ be two finite subsets of $S$. There exists a finite subset $\bar{\bar F}\subset S$ such that $F\subset\bar{\bar F},$ $\ker \psi_{F, \bar{\bar F}}=\ker\psi_F,$ and $\ker\psi_{F', \bar{\bar F}}=\ker \psi_{F'}$. Then $\Ima \psi_{F, \bar{\bar F}}$ is isomorphic to $\Ima\psi_F$ and is equal to $\Ue(\frak{sl}(F))\cdot v_{\bar{\bar F}}$, and $\Ima\psi_{F', \bar{\bar F}}$ is isomorphic to $\Ima\psi_{F'}$ and is equal to $\Ue(\frak{sl}(F'))\cdot v_{\bar{\bar F}}$. This defines an embedding of $\Ima\psi_F$ to $\Ima\psi_{F'}$ such that $\psi_F(w_F)\mapsto \psi_{F'}(w_{F'})$.

The limit of the direct system of such morphisms over all finite subsets $F$ of $S$ defines an $\frak{sl}(\infty)$-module $\tilde L(f)$. Clearly, the direct limit of the vectors $\psi_F(w_F)$ is a highest weight vector of weight $\lambda_f$ in $\tilde L(f)$. Denote this vector by $\tilde v$. We claim that $\tilde L(f)$ is isomorphic to $L(f)$. For the proof we provide two $\frak{sl}(\infty)$-morphisms $L(f)\to \tilde L(f)$ and $\tilde L(f)\to L(f)$ such that $v\mapsto \tilde v$ and $\tilde v\mapsto v$ respectively.

The morphism $\tilde L(f)\to L(f)$ arises from the fact that $\tilde L(f)$ is a highest weight module with highest weight $\lambda_f$. We may assume that under this morphism $\tilde v$ goes to $v$ (in general, $\tilde v$ maps to some vector proportional to $v$). Now we construct a morphism $L(f)\to \tilde L(f)$. For any set $F$ we pick $\bar F$ as described above and consider the chain
$$M(f|_F)\to\widehat L(f|_F)\to\widehat L(f|_{\bar F})\to L(f|_{\bar F})$$
of $\frak{sl}(F)$-morphisms whose composition is $\psi_F$. This defines an $\frak{sl}(F)$-morphism $\widehat L(f|_F)\to \Ima\psi_F$. By passing to the direct limit, we obtain the desired $\frak{sl}(\infty)$-morphism $L(f)\to \tilde L(f)$.

Since the $\frak{sl}(F)$-submodule of $\tilde L(f)$ generated by the image of $v$ in $\tilde L(f)$ is isomorphic to $\Ima\psi_F$, the proposition is proved.\end{proof}

Any compatible partition $S_1\sqcup...\sqcup S_t$ of $S$ defines a parabolic subalgebra of $\frak{sl}(S)$: this is the algebra $\frak p$ with root decomposition $$\frak h_S\oplus\bigoplus\limits_{i\prec j\mbox{~or~}i\ne j\in S_k\mbox{~for~some~}k}\mathbb Ce_{ij}.$$
We set $\frak p=\frak l \supsetplus\frak n,$ where $$\frak l:=\frak h_S+\bigoplus\limits_{i\in S_k}\frak{sl}(S_i),\hspace{10pt}\frak n:=\bigoplus\limits_{e_{ij}\notin\frak l, i\prec j}\mathbb Ce_{ij}.$$
Set also $\frak n^-:=\bigoplus\limits_{e_{ij}\notin\frak l, j\prec i}\mathbb Ce_{ij}.$

\begin{proposition}\label{Lpind}Let $S_1\sqcup...\sqcup S_t$ be a compatible partition of $S$, and $f\in\mathbb C^S$ be a function such that
$$f(k')-f(l')\notin \mathbb Z$$for all  $k'\in S_i, l'\in S_l$ where $i<l$. Then $L(f)$ is isomorphic to $$\Ue(\frak{sl}(S))\otimes_{\Ue(\frak p)}L(f)^{\frak n},$$ where $L(f)^{\frak n}$ stands for the $\frak n$-invariants of $L(f)$. Moreover, as an $\frak{sl}(S_1)\oplus...\oplus\frak{sl}(S_t)$-module, $L(f)^{\frak n}$ is isomorphic to
$$L(f|_{S_1})\otimes...\otimes L(f|_{S_t}).$$\end{proposition}
\begin{proof}We set $L_\frak p:=\Ue(\frak l)\cdot v$. Standard arguments show that $L_\frak p$ is a simple $\frak l$-module and that $L_\frak p=L(f)^{\frak n}$. Therefore we have a natural surjective $\frak{sl}(S)$-morphism
$$\alpha: \Ue(\frak{sl}(S))\otimes_{\Ue(\frak p)}L_\frak p\to L(f).$$
We claim that $\alpha$ is an isomorphism. For this it suffices to show that
$$\beta: \Ue(\frak n^-)\otimes_{\mathbb C}L_\frak p\to L(f)\hspace{10pt}(n\otimes l\mapsto nl)$$is injective. However, the injectivity of $\beta$ follows from the fact established above that the natural map $\widehat L(f|_F)\to L(f|_{\bar F})$ is an injection for any finite subset $F\subset S$.

Next, one notes that the simplicity of $L_\frak p$ as a $\frak l$-module implies its simplicity as an $[\frak l, \frak l]$-module. This follows from the fact that any $\frak h$-weight space of $L_\frak p$ is also an $(\frak h\cap[\frak l, \frak l])$-weight space. Then, since
$L(f|_{S_1})\otimes...\otimes L(f|_{S_t})$ and $L_\frak p$ are simple $([\frak l, \frak l]\cap \frak b_S)$-highest weight $(\frak{sl}(S_1)\oplus...\oplus\frak{sl}(S_t)=[\frak l, \frak l])$-modules with the same highest weight, they are isomorphic.
\end{proof}

We say that an $\frak{sl}(S)$-module $M$ is {\it integrable}, if
$$\dim (\Ue(\frak g)\cdot m)<\infty$$for any $m\in M$ and any
finite-dimensional Lie subalgebra $\frak g\subset\frak{sl}(S)$.
\begin{proposition}\label{Ldom-} Let $S_1\sqcup ...\sqcup S_t$ be a partition of $S$ compatible with $\prec$. Then the $\frak{sl}(S)$-module $L(f)$ is $\frak{sl}(S_k)$-integrable if and only if
$f|_{S_k}\in \mathbb C^{S_k}$ is dominant.\end{proposition}





\begin{proof} Assume that $L(f)$ is $\frak{sl}(S_k)$-integrable. Then the $(\frak{sl}(2)\cong\frak{sl}(\{i, j\}))$-module $\Ue(\frak{sl}(\{i, j\}))\cdot v$ is integrable for $i, j\in S_k$. Hence $f(i)-f(j)\in\mathbb Z$ for $i, j\in S_k$ by a well-known statement about $\frak{sl}(2)$.

Now we wish to prove that if $f|_{S_k}$ is dominant, then $L(f)$ is $\frak{sl}(S_k)$-integrable. Clearly, it suffices to show that $\widehat L(f|_F)$ is $\frak{sl}(S_k\cap S)$-integrable for any finite subset $F\subset S$. According to Proposition~\ref{Pfinf}, $\widehat L(f|_F)\cong\Ima\psi_F$. The fact that $\Ima\psi_F$ is $\frak{sl}(S_k\cap S)$-integrable follows from the well-known fact, concerning modules over finite-dimensional Lie algebras, that, for any finite subset $F'\subset S$, the $\frak{sl}(F')$-module $L(f|_{F'})$ is $\frak{sl}(S_k\cap F')$-integrable.






\end{proof}
\begin{corollary}\label{Ldom} The $\frak{sl}(S)$-module $L(f)$ is integrable if and only if
$f\in \mathbb C^S$ is dominant.\end{corollary}
\begin{corollary}\label{Ppint}Assume that $f$ is locally constant with respect to a compatible partition $S_1\sqcup...\sqcup S_n$ of $S$. Then $L(f)$ is an integrable $\frak{sl}(S_i)$-module for any $i\le n$.\end{corollary}

\subsection{Ideals of $\Ue(\frak{sl}(\infty))$}
Let $I$ be an ideal of $\Ue(\frak{sl}(\infty))$. Under an ideal we always mean a two-sided ideal. Fix an exhaustion
\begin{equation}\frak{sl}(2)\hookrightarrow\frak{sl}(3)\hookrightarrow...\hookrightarrow\frak{sl}(n)\hookrightarrow\frak{sl}(n+1)\hookrightarrow...\label{Eex}\end{equation}of
$\frak{sl}(\infty)$. Then $I=\varinjlim\limits_{n\ge2}(I\cap\Ue(\frak{sl}(n)))$. Set $I_n:=I\cap\Ue(\frak{sl}(n))$. Let $\Var I_n\subset\frak{sl}(n)^*$ be the associated variety of~$I_n$. By identifying $\frak{sl}(n)$ and $\frak{sl}(n)^*$ via the Killing form we can assume that $\Var I_n\subset\frak{sl}(n)$.

For any positive integer $r$ we introduce the varieties
\begin{center}$\frak{sl}(n)^{\le r}:=\{x\in\frak{sl}(n)\mid
\exists\lambda\in\mathbb C$ such that $\APrk (x-\lambda)\le
r\}$,\end{center}where $\lambda$ is understood as a scalar $n\times n$-matrix. One can easily see that
$\frak{sl}(n)^{\le r}$ is an $SL(n)$-stable subvariety of $\frak{sl}(n)$.

The following theorem reproduces the claim of~\cite[Corollary~6.2~b)]{PP} for $\frak{sl}(\infty)$.
\begin{theorem}\label{Tuslrk}For any nonzero ideal $I\subset\Ue(\frak{sl}(\infty))$ such that  $\Var I_n\ne 0$ for some $n$, there exists a positive integer $r$ such that $\Var I_n=\frak{sl}(n)^{\le r}$ for any $n\ge2$.\end{theorem}

\subsection{Integrable ideals and coherent local systems}\label{SSzhi}
We say that an ideal $I\subset\Ue(\frak{sl}(\infty))$ is {\it integrable}, if $I$ is the annihilator of an integrable $\frak{sl}(\infty)$-module. Integrable ideals are closely connected with coherent local systems of modules which we define next.








Let $\Irr_n$ denote the set of isomorphism classes of simple finite-dimensional $\frak{sl}(n)$-modules.
\begin{definition}{\it A coherent local system of modules} (further shortened as {\it c.l.s.}) for $\frak{sl}(\infty)=\varinjlim\frak{sl}(n)$ is a collection of subsets \begin{center}$\{Q_n\}_{n\in\mathbb Z_{\ge2}}\subset\prod_{n\in\mathbb Z_{\ge2}}\Irr_n$\end{center} such that $Q_m=\APbre{Q_n}_m$ for any $n>m$, where $\APbre{Q_n}_m$ denotes the set of isomorphism classes of all simple $\frak{sl}(m)$-constituents of the $\frak{sl}(n)$-modules from $Q_n$.\end{definition}

A.~Zhilinskii~\cite{Zh2, Zh3} has classified c.l.s. for $\frak{sl}(\infty)$ and more generally for any
locally simple Lie algebra. Moreover, if $Q$ is a
c.l.s., then $$I(Q_m):=\cap_{z\in Q_m}\Ann_{\Ue(\frak{sl}(m))}z\subset\cap_{z\in
Q_n}\Ann_{\Ue(\frak{sl}(n))}z=:I(Q_n)$$ for any $n>m$. Therefore
$\cup_n(\cap_{z\in Q_n}\Ann_{\Ue(\frak{sl}(n))}z)=\cup_nI(Q_n)$ is an ideal of
$\Ue(\frak{sl}(\infty))$; we denote it by $I(Q)$. It follows from {~\cite[Lemma 1.1.2]{Zh2}} that $I(Q)$ is
integrable.

It turns out that Zhilinskii's classification of c.l.s. yields a
classification of integrable ideals of $\Ue(\frak{sl}(\infty))$. In this paper we present only the classification of c.l.s. For the classification of integrable ideals see~\cite[Theorem~7.9]{PP}.

A c.l.s. $Q$ is {\it irreducible} if $Q\ne Q'\cup Q''$ with
$Q'\not\subset Q''$ and $Q''\not\subset Q'$. The following proposition clarifies the role of the irreducible c.l.s.

\begin{proposition}\label{Pintm}a) If $Q$ is an irreducible c.l.s., then $I(Q)$ is the annihilator of a simple $\frak{sl}(\infty)$-module. In particular, $I(Q$) is primitive and hence prime.

b) If $I$ is an integrable prime ideal of $\Ue(\frak{sl}(\infty))$, then $I=I(Q)$ for an irreducible c.l.s. $Q$. \end{proposition}
\begin{proof}Part a) follows directly from{~\cite[Lemma 1.1.2]{Zh2}}. Part b) is a consequence of~\cite[Theorem 7.9]{PP}.\end{proof}


Fix $n$. The set $\Irr_n$ is parametrized by the lattice of integral dominant weights of $\frak{sl}(n)$. Let $z_1, z_2$ be
isomorphism classes of simple $\frak{sl}(n)$-modules with respective
highest weights $\lambda_1, \lambda_2$. We denote by $z_1z_2$ the
isomorphism class of the simple module with highest weight
$\lambda_1+\lambda_2$. If $S_1, S_2\subset\Irr_n$ we
set\begin{center}$S_1S_2:=\{z\in\Irr_n\mid z=z_1z_2$ for some
$z_1\in S_1$ and $z_2\in S_2\}$.\end{center}If $Q'$, $Q''$ are
c.l.s., we denote by $Q'Q''$ the smallest c.l.s. such that
$(Q')_i(Q'')_i\subset(Q'Q'')_i$.
By~\cite{Zh2}\begin{center}$(Q')_n(Q'')_n=(Q'Q'')_n$.\end{center}
\subsubsection{Zhilinskii's classification of c.l.s.}In this
subsection we reproduce Zhilinskii's classification of
c.l.s. for $\frak{sl}(\infty)$~\cite{Zh1}.
Any integrable $\frak{sl}(\infty)$-module $M$ determines a c.l.s.
$Q:=\{Q_n\}_{n\in\mathbb Z_{>0}}$, where
\begin{center}$Q_n:=\{z\in\Irr_n\mid$Hom$(z, M)\ne 0$\}.\end{center}
We set $\cls(M):=Q$. Moreover, $\Ann_{\Ue(\frak{sl}(\infty))} M=I(\cls(M))$. We construct an irreducible c.l.s. as the c.l.s. of some explicitly given integrable $\frak{sl}(\infty)$-module.

Let $V(\infty)$ denote a vector space with basis $\{e_i\}_{i\in\mathbb Z_{>0}}$. We endow $V(\infty)$ with an action of $\frak{sl}(\infty)$ by putting$$e_{ij}\cdot e_k=e_i\delta_{jk},\hspace{10pt}(e_{ii}-e_{jj})\cdot e_k=e_i\delta_{ik}-e_j\delta_{jk}~\mbox{~for~}i, j, k\in\mathbb Z_{>0}.$$In this way $V(\infty)$ becomes a simple integrable $\frak{sl}(\infty)$-module, and we call it the {\it natural $\frak{sl}(\infty)$-module}. By $V(\infty)_*$ we denote the restricted dual to $V(\infty)$, i.e., the $\frak{sl}(\infty)$-submodule of $V(\infty)^*$ spanned by
the vectors~$\{e_i^*\}_{~i\in\mathbb Z_{>0}}$ which satisfy
$$e_i^*(e_j)=\delta_{ij}.$$

Any irreducible c.l.s. $Q$ for $\frak{sl}(\infty)$ is a product of
the following {\it basic} c.l.s.:
\begin{center}$\mathcal E:=\cls(\Lambda^\cdot V(\infty)),\hspace{10pt}\mathcal L_p:=\cls(\Lambda^p V(\infty)),\hspace{10pt} \mathcal
L_p^\infty:=\cls(${\bf S}$^\cdot(V(\infty)\otimes\mathbb C^p)),$\\$\mathcal R_q:=\cls(\Lambda^qV(\infty)_*),\hspace{10pt} \mathcal R_q^\infty:=\cls(${\bf
S}$^\cdot(V(\infty)_*\otimes\mathbb C^q)),
\hspace{10pt}\mathcal E^\infty$ (all modules),\end{center} where $p,
q\in\mathbb Z_{\ge0}$. More precisely, any irreducible c.l.s. is expressed uniquely
as\begin{equation}(\mathcal L_v^\infty\mathcal L_{v+1}^{x_{v+1}}\mathcal
L_{v+2}^{x_{v+2}}...\mathcal L_n^{x_n})~~\mathcal E^m~~(\mathcal
R_w^\infty\mathcal R_{w+1}^{z_{w+1}}\mathcal
R_{w+2}^{z_{w+2}}...\mathcal R_l^{z_l}),\label{Ecfcls}\end{equation}where
\begin{center}$m, n, l, v, w, x_i, z_j\in\mathbb Z_{\ge0}$.\end{center}Here, if $v=0$ (respectively $w=0$), then
$\mathcal L_v^\infty$ (respectively $\mathcal R_w^\infty$) is
assumed to be the identity (the c.l.s. consisting of the isomorphism
class of the trivial 1-dimensional module at all levels).
In~\cite{Zh2} the above formulas are called the {\it unique
factorization property}.


\subsection{C.l.s. of simple integrable highest weight modules}We start with the following definition.
\begin{definition}A c.l.s. $Q$ is {\it of finite type} if $Q_n$ is finite for any $n$.\end{definition}
One can easily check that the irreducible c.l.s. of finite type are precisely the c.l.s. of the form~(\ref{Ecfcls}) with $v=w=0$.

Let $f\in\mathbb C^{\mathbb Z_{>0}}$ be an integral function. We assume that a linear order on $\mathbb Z_{>0}$ is fixed and therefore we use the notations of Section~\ref{SSmi} for $S=\mathbb Z_{>0}$.

Proposition~\ref{Lftfmv} below implies that if $|f|=\infty$ then $\cls(L(f))=\mathcal E^\infty$. One can
check that the proof of Proposition~\ref{Lftfmv} is independent from
the current discussion. If $|f|<\infty$, there are two values $a, b\in\mathbb C$ of $f$ such that $a-b\in\mathbb Z_{\ge0}$ is maximal. We set $s:=a-b$. For any nonnegative integer $c\le s$ we denote by $d_{c}$ the multiplicity of the value $b+c$ of $f$ (note that $d_{c}\in\mathbb Z_{\ge0}\cup+\infty$). Let $p$ be the smallest integer such that
$d_{p}=+\infty$, and $q$ be the largest integer such that
$d_{q}=+\infty$ (if $d_{c}$ is finite for all $0\le c\le b-a$, we put $p=q=0$).
\begin{proposition}\label{Lexpl}\label{CexY}

a) For a $\prec$-dominant function $f\in\mathbb C^{\mathbb Z_{>0}}$ with $|f|<\infty$, we have \begin{equation}\cls(L(f))=\mathcal L_{d_0}\mathcal L_{d_0+d_1}...\mathcal L_{d_0+d_1+...+d_{p-1}}\mathcal E^{(q-p)}\mathcal R_{d_s}\mathcal R_{d_s+d_{s-1}}...\mathcal R_{d_s+...+d_{q+1}}.\label{Eclsft}\end{equation}

b) A c.l.s. of the form~(\ref{Eclsft}) is of finite type.

c) Let $\frak b^0$ be a fixed ideal Borel subalgebra of $\frak{sl}(\infty)$. Then any irreducible c.l.s. of finite type is equal to $\cls(L_{\frak b^0}(f^0))$ for an appropriate $\frak b^0$-dominant function $f^0\in\mathbb C^{\mathbb Z_{>0}}$.\end{proposition}
\begin{proof}First we prove part a). Recall that $L(f)=\varinjlim\limits_{n\ge 2} L(f|_{\{1,..., n\}})$. Thus the coherent local system $\cls(L(f))$ is determined by the highest weights $\lambda_n$ of the finite-dimensional $\frak{sl}(n)$-modules $L(f|_{\{1,..., n\}})$. Such local systems have been considered by Zhilinskii~\cite{Zh1} and he provides an explicit algorithm which assigns to $\{\lambda_i\}$ a c.l.s. of the form~(\ref{Eclsft}). This implies a).

b) It is clear that any c.l.s. of the form~(\ref{Eclsft}) is a c.l.s. of finite type.

c) The ideal subalgebra $\frak b^0$ defines a partition $S_1\sqcup S_2\sqcup S_3$ of $\mathbb Z_{>0}$ with fixed order preserving bijections $\mathbb Z_{>0}\to S_1,~\mathbb Z_{<0}\to S_3$. We denote the image of $k\in\mathbb Z_{>0}$ in $S_1$ by $k^1$, and the image of $-k\in\mathbb Z_{<0}$ in $S_3$ by $k^3$.

It is clear that any c.l.s. $Q$ of the form~(\ref{Ecfcls}) with $v=w=0$ can be presented in the form~(\ref{Eclsft}) for suitable integers $p, q, d_0, d_1,..., d_{p-1}, d_s, d_{s-1},..., d_{q+1}$. Moreover, $\cls(L_{\frak b^0}(f))=Q$ for the $\frak b^0$-dominant function $f\in\mathbb C^{\mathbb Z_{>0}}$ defined as follows: $$f(i):=\begin{cases}p&\mbox{~if~}i\in S_2\mbox{~or~}i=k^1\mbox{~for~}k>d_0+...+d_{p-1};\\
q &\mbox{~if~}i=k^3\mbox{~for~}k>d_s+...+d_{q+1};\\
j &\mbox{~if~}i=k^1\mbox{~and~}d_0+...+d_{j-1}<k\le d_0+...+d_j\mbox{~for~}j<p;\\
j &\mbox{~if~}i=k^3\mbox{~and~}d_s+...+d_{j+1}<k\le d_s+...+d_j\mbox{~for~}j>q.\\
\end{cases}$$
\end{proof}

\section{Statements of Main Results}\label{Sstate}
\begin{theorem}\label{1T} Let $\frak b\supset\frak h$ be a splitting Borel subalgebra of $\frak{sl}(\infty)$, and $f\in\mathbb C^{\mathbb Z_{>0}}$. Then $\Ann_{\Ue(\frak{sl}(\infty))} L_{\frak b}(f)\ne 0$ if and only if $f$ is almost integral and locally constant with respect to the linear order defined by $\frak b$.\end{theorem}
\begin{theorem}\label{2T} The following conditions on a nonzero ideal $I$ of $\Ue(\frak{sl}(\infty))$ are equivalent:

--- $I=\Ann_{\Ue(\frak{sl}(\infty))} L_{\frak b}(f)$ for some splitting Borel subalgebra $\frak b\supset\frak h$ and some function $f\in\mathbb C^{\mathbb Z_{>0}}$;

--- $I$ is a prime integrable ideal of $\Ue(\frak{sl}(\infty))$;

--- $I=\Ann_{\Ue(\frak{sl}(\infty))} L_{\frak b^0}(f^0)$ for some $f^0\in\mathbb C^{\mathbb Z_{>0}}$, where $\frak b^0$ is any fixed ideal Borel subalgebra.
\end{theorem}
\begin{proposition}\label{Pp} If $\frak b$ is a nonideal Borel subalgebra then there exists a prime integrable ideal $I$ which does not arise as the annihilator of a simple $\frak b$-highest weight $\frak{sl}(\infty)$-module.\end{proposition}

We split the proof of Theorem~\ref{1T} into two parts:

a) if $\Ann_{\Ue(\frak{sl}(\infty))} L_{\frak b}(f)\ne0$, then $f$ is almost integral and locally constant;

b) if $f$ is almost integral and locally constant with respect to the order defined by $\frak b$, then $$\Ann_{\Ue(\frak{sl}(\infty))} L_{\frak b}(f)\ne0.$$
Parts a) and b) of Theorem~\ref{1T} are proved in Sections~\ref{Spr1ta} and~\ref{Spr1tb} respectively. Theorem~\ref{2T} and Proposition~\ref{Pp} are proved in Section~\ref{Spr2t}.

\section{Proof of Theorem~\ref{1T}~a)}\label{Spr1ta}
To prove Theorem~\ref{1T}~a), we fix a Borel subalgebra $\frak b\supset\frak h$ of $\frak{sl}(\infty)$ and hence an order $\prec$ on $\mathbb Z_{>0}$. Throughout Sections~\ref{Spr1ta} and~\ref{Spr1tb} we suppress the dependence from $\frak b$ and $\prec$ in all notation. We set $I(f):=\Ann_{\Ue(\frak{sl}(\infty))} L(f)$ for any $f\in\mathbb C^{\mathbb Z_{>0}}$. Sometimes we consider the finite-dimensional Lie algebra $\frak{sl}(n)$. In this case the fixed order $\{1,..., n\}$ is the standard order, and $I(f)\subset\Ue(\frak{sl}(n))$ is the annihilator of the simple $\frak{sl}(n)$-module with highest weight $\lambda_f$ for $f\in\mathbb C^n$.

Theorem~\ref{1T}a) follows from Propositions~\ref{Lftfmv},~\ref{Lfdcf} and~\ref{Lflc} below.
\begin{proposition}\label{Lftfmv}Let $f\in\mathbb C^{\mathbb Z_{>0}}$. If $I(f)$ is nonzero, then $|f|<\infty$.\end{proposition}
\begin{proposition}\label{Lfdcf}Let $f\in\mathbb C^{\mathbb Z_{>0}}$. If $I(f)$ is nonzero, then $f$ is almost integral.\end{proposition}
\begin{proposition}\label{Lflc}Let $f\in\mathbb C^{\mathbb Z_{>0}}$. If $I(f)$ is nonzero, then $f$ is locally constant with respect to $\prec$.\end{proposition}
We prove these propositions consecutively in Sections~\ref{SSpffmv},~\ref{SSicf} and~\ref{SSflc}. Clearly, Proposition~\ref{Lftfmv} follows from Proposition~\ref{Lflc}, however we require Proposition~\ref{Lftfmv} for the proof of Proposition~\ref{Lflc}. Propositions~\ref{Lfdcf} and~\ref{Lflc} rely on a version of the Robinson-Schensted algorithm which we present in Section~\ref{SSrsa}.


\subsection{Proof of Proposition~\ref{Lftfmv}}\label{SSpffmv}
We start with some notation. Let $n\ge2$ be a positive integer. For any ideal $I\subset\Ue(\frak{sl}(n))$ we denote by
gr$I\subset\Sa^\cdot(\frak{sl}(n))$ the associated graded ideal. By $\Var(I)\subset\frak{sl}(n)^*$ we denote the set of zeros of gr$I$.

The radical ideals of the center
$\ZUe(\frak{sl}(n))$ of $\Ue(\frak{sl}(n))$ are in one-to-one
correspondence with $\frak G_n$-invariant closed subvarieties of $\frak
h_n^*$, where $\frak h_n$ is a fixed Cartan subalgebra of $\frak{sl}(n)$ and $\frak G_n$ is the symmetric group on $n$ letters. Let
$I$ be an ideal of $\Ue(\frak{sl}(n))$. Then $\ZVar(I)$ denotes the subvariety of $\frak h_n^*$
corresponding to the radical of the ideal $I\cap\ZUe(\frak{sl}(n))$ of $\ZUe(\frak{sl}(n))$. If
$\{I_t\}$ is any collection of ideals in $\Ue(\frak{sl}(n))$, then
\begin{equation}\ZVar(\cap_tI_t)=\overline{\cup_t\ZVar(I_t)}\label{Ezv},\end{equation} where here and below bar indicates Zariski closure.

Let $\phi: \{1,..., n\}\to \mathbb Z_{>0}$ be an injective map.
Slightly abusing notation, we denote 
by $\phi$ the induced homomorphism\begin{center}$\phi:
\Ue(\frak{sl}(n))\to\Ue(\frak{sl}(\infty))$.\end{center}
By $\APinj(n)$ we denote the set of injective maps from $\{1,..., n\}$ to $\mathbb Z_{>0}$, and by $\injo(n)$ the set of order preserving maps from $\{1,..., n\}$ to $\mathbb Z_{>0}$ with respect to the standard order on $\{1,..., n\}$ and the order $\prec$ on $\mathbb Z_{>0}$.

By $\frak{sl}(\phi)$ we denote $\frak{sl}(\Ima\phi)\subset\frak{sl}(\infty)$. For any $f\in \mathbb C^{\mathbb Z_{>0}}$ we set $f_\phi:=f\circ\phi$. Then $M(f_\phi):=M_{\frak b\cap\frak{sl}(\phi)}(f_\phi)$ and $L(f_\phi):=L_{\frak b\cap\frak{sl}(\phi)}(f_\phi)$ are well defined $(\frak b\cap\frak{sl}(\phi))$-highest weight $\frak{sl}(\phi)$-modules. If $f$ is dominant and $\phi\in\injo(n)$, then $f_\phi$ is $(\frak b\cap\frak{sl}(\phi))$-dominant. 

Let $\phi\in\injo(n)$ and
$\widetilde{\Ver}(f)$ be any quotient of $\Ver(f)$. It is well known
that\begin{center}$\ZVar(\Ann_{\Ue(\frak{sl}(\phi))}
\Ver(f_\phi))=\ZVar(\Ann_{\Ue(\frak{sl}(\phi))}\widetilde{\Ver}(f_\phi))=\ZVar(\Ann_{\Ue(\frak{sl}(\phi))}L(f_\phi))=\frak
G_n(\rho_n+\lambda_{f_\phi})$,\end{center} where $\rho_n\subset\frak
h_n^*$ is the half-sum of positive roots. 

Let $\frak g$ be a Lie algebra. {\it The adjoint group of $\frak g$} is the subgroup of $\APAut\frak g$ generated by the exponents of all nilpotent elements of $\frak g$. We denote this group by $\Adj\frak g$.

\begin{lemma}\label{Lconj}Let $\phi_1: \frak k\to\frak g$ and $\phi_2: \frak k\to\frak g$ be two $\Adj\frak g$-conjugate morphisms of Lie
algebras. Let $I$ be a two-sided ideal of $\Ue(\frak g)$. Then
\begin{center}$\phi_1^{-1}(I)=\phi_2^{-1}(I)$.\end{center}\end{lemma}
\begin{proof}The adjoint action of $\frak g$ on $\Ue(\frak g)$ extends
uniquely to an action of $\Adj\frak g$ on $\Ue(\frak g)$. The ideal
$I$ is $\frak g$-stable and thus is $\Adj\frak g$-stable. Let $g\in
\Adj\frak g$ be such that $\phi_1=g\circ\phi_2$.
Then\begin{center}$\phi_1^{-1}(g(i))=\phi_2^{-1}(i)$\end{center}for any $i\in I$. Hence,
\begin{center}$\phi_1^{-1}(I)=\phi_2^{-1}(I)$.\end{center}\end{proof}

\begin{proof}[Proof of Proposition~\ref{Lftfmv}]Let $I(f)\ne 0$. Assume to the contrary that there exist $i_1,...,i_s,...\in
\mathbb Z_{>0}$ such that $$f(i_1),...., f(i_s),...$$are pairwise distinct
elements of $\mathbb C$. As $I(f)\ne\Ue(\frak{sl}(\infty))$, there exists a positive integer $n$ and an injective
map $\phi: \{1,..., n\}\to \mathbb Z_{>0}$ such
that\begin{center}$I_\phi:=I(f)\cap\Ue(\frak{sl}(\phi))\ne
0$,\end{center} or equivalently
\begin{equation}\Ue(\frak{sl}(n))\supset \phi^{-1}(I(f))=\phi^{-1}(I_\phi)\ne0.\label{Eslp}\end{equation}

Let $\psi\in\APinj(n)$ be another map. Since $\phi$ and $\psi$ are conjugate via the adjoint group of $\frak{sl}(\infty)$, we have
\begin{equation}\phi^{-1}(I(f))=\psi^{-1}(I(f))\ne0\label{Econj}.\end{equation}
This means that $\phi^{-1}(I(f))$ depends on $n$ and $f$ but not on $\phi$, and we set $I_n:=\phi^{-1}(I(f))$.

Assume now that $\phi\in\injo(n)$. Then the
highest weight space of the $\frak{sl}(\infty)$-module $L(f)$ generates a highest weight $\frak{sl}(\phi)$-submodule
$\widehat{L}(f_\phi)$. Clearly,
\begin{center}$\Ann_{\Ue(\frak{sl}(\phi))}L(f)\subset
\Ann_{\Ue(\frak{sl}(\phi))}\widehat{L}(f_\phi).$\end{center}Therefore,\begin{center}$I_n\subset\cap_{\phi\in
\injo(n)}\Ann_{\Ue(\frak{sl}(n))}\widehat{L}(f_\phi)$\end{center}
and\begin{center}$I_n\cap\ZUe(\frak{sl}(n))\subset\cap_{\phi\in
\injo(n)}(\Ann_{\Ue(\frak{sl}(n))} \widehat{L}(f_\phi)\cap \ZUe(\frak{sl}(n))).$\end{center}Hence,
according to~(\ref{Ezv}) we have
\begin{center}$\overline{\cup_{\phi\in \injo(n)}\frak
G_n(\rho_n+\lambda_{f_\phi})}=\frak G_n(\rho_n+\overline{\cup_{\phi\in
\injo(n)}\lambda_{f_\phi}})\subset \ZVar(I_n)$.\end{center}

We claim that
\begin{center}$\overline{\frak
G_n(\cup_{\phi\in\injo(n)}\lambda_{f_\phi}})=\frak
h_n^*,$
\end{center}and thus
that\begin{equation}\ZVar(I_n)=\frak
h_n^*.\label{Efd1}\end{equation} Our claim is
equivalent to the equality \begin{center}$\overline{\frak
G_n(\cup_{\phi\in\injo(n)}\lambda_{f_\phi}})=\overline{(\cup_{\phi\in\APinj(n)}\lambda_{f_\phi}})=\frak
h_n^*$
\end{center} which is
implied by the following
equality:\begin{equation}\overline{(\cup_{\phi\in\APinj(n)}f_\phi})=\mathbb C^n\label{Efd}.\end{equation}

We now prove~(\ref{Efd}) by
induction. The inclusion $\{1,..., n-j\}\to\{1,..., n\}$ induces a restriction map $$res: \mathbb C^n\to \mathbb C^{n-j}.$$ Denote by $f_\psi*$ the preimage of $f_\psi$ under
$res$ for $\psi\in\APinj(n-j)$. 
We will show that 
\begin{equation}f_\psi*\subset\overline{\cup_{\phi\in
\APinj(n)}f_\phi}\label{Efps}\end{equation}for any  $j\le n$ and any map $\psi\in\APinj(n-j)$. This holds trivially for $j=0$. Assume that it also holds for $j$.
Fix $\psi\in \APinj(n-j-1)$ and set
$$(\psi\times k)(l):=\begin{cases}\psi(l)&\mbox{if}~l\le n-j-1\\i_k&\mbox{if}~l=n-j\end{cases}.$$
It is clear that there exists $s\in\mathbb Z_{\ge1}$ such that
$$(\psi\times k)\in\APinj(n-j)$$ for any $k\in\mathbb Z_{\ge s}$.
Moreover, $f_{\psi\times k_1}\ne f_{\psi\times k_2}$ for any $k_1\ne k_2$.
Therefore$$\overline{\cup_{k\in\mathbb Z_{\ge s}}f_{\psi\times
k}*}=f_\psi*,$$which yields~(\ref{Efps}).

For $j=n$,~(\ref{Efps}) yields $\mathbb C^n\subset \overline{\cup_{\phi\in
\APinj(n)}f_\phi}$, consequently~(\ref{Efd}) holds. Then~(\ref{Efd1}) holds also, hence


$$I_n\cap\ZUe(\frak{sl}(n))=0.$$
It is a well known fact that an ideal of $\Ue(\frak{sl}(n))$ whose intersection with $\ZUe(\frak{sl}(n))$ equals zero is the zero ideal~\cite[Proposition 4.2.2]{Dix}. Therefore, we have a contradiction with~(\ref{Eslp}), and the proof is complete.
\end{proof}

\subsection{Algorithm for $\frak{sl}(n)$}\label{SSrsa}
According to Duflo's Theorem, any primitive ideal of
$\Ue(\frak{sl}(n))$ is the annihilator of some simple highest weight
module, i.e., any primitive ideal is of the form $I(f)$ for some $f\in\mathbb C^n$. 
 The associated variety 
of $I(f)$ is the closure of a certain nilpotent
coadjoint orbit $\EuScript O(f)$ of $\frak{sl}(n)$~\cite{Jo4}. To $\EuScript O(f)\subset\frak{sl}(n)^*$ one assigns a partition $p(f)$ of $n$ as follows. 
One first represents $\EuScript O(f)$ by a nilpotent element $x\in\frak{sl}(n)$. Then $p(f)$ is the partition conjugate to the partition arising from the sizes of Jordan blocks of $x$ considered as a linear operator on the natural representation of $\frak{sl}(n)$.

We now describe the algorithm which computes $p(f)$. This is a modification of the Robinson-Schensted algorithm, see~\cite[Theorem A on p. 52]{Knu}.

Let $f\in\mathbb C^n$ be a function. 

Step 1) Set $f^+:= (f(1), f(2)-1,..., f(n)-n+1)$.

Step 2) Introduce an equivalence relation $\sim$ on $\{1,..., n\}$:
\begin{center}$i\sim j$ if and only if $f(i)-f(j)\in\mathbb Z$.\end{center}
Let $t$ be the number of equivalence classes for $\sim$, and let $n_1,..., n_t$ be the cardinalities of the respective equivalence classes.

Step 3) Consider $f^+$ as a function $f^+:\{1,.., n\}\to\mathbb C$. The restriction of $f^+$ to the equivalence classes of Step 2) defines subsequences $\seq_1(f^+),
\seq_2(f^+),..., \seq_t(f^+)$ of respective lengths $n_1,...,
n_t$.

Step 4) Fix $i$. Note that the elements of $\seq_i(f^+)$ are linearly ordered as their pairwise differences are integers. Since the elements of $\seq_i(f^+)$ are not necessary pairwise distinct, we modify the above linear order by letting $f^+(m)\triangleright f^+(k)$ if $m>k$ and $f^+(m)=f^+(k)$. In this way we introduce a new linearly ordered set $\widetilde{\seq_i(f^+)}$ of cardinality $n_i$.




Step 5) Apply the Robinson-Schensted algorithm to the linearly ordered sets $\widetilde{\seq_i(f^+)}$ from Step 4) to produce partitions $p_i$ of $n_i$.

Step 6) Consider the partitions $p_1, p_2,..., p_t$ as a partition $\RS(f)$ of $n$.

\begin{proposition}\label{Ppart} Let $f\in\mathbb C^n$ be a function. Then $p(f)=\RS(f)$.\end{proposition}


\begin{proof}
This statement is contained in the work of A.~Joseph, so all we need to do is to translate Joseph's result to the language which we use in this paper. For any $f'\in\mathbb C^n$ set $(f')^\#:= (f'(1), f'(2)+1,..., f'(n)+n-1)$.

We note first that \begin{equation}I(f)=I((\seq_1(f^+),...,\seq_t(f^+))^\#).\label{Eww0}\end{equation}  This is a translation of the equality \begin{equation}{\rm J}({\rm w}_1{\rm w}_2\lambda)={\rm J}({\rm w}_2\lambda)\label{Eww}\end{equation} for appropriate choices of Weyl group elements ${\rm w}_1, {\rm w}_2$, as stated at the bottom of the first page of~\cite{Jo3} (the equality~(\ref{Eww}) uses the notation of A.~Joseph which is slightly different from ours). Thus we can assume further that $f^+=(\seq_1(f^+),\seq_2(f^+),..., \seq_t(f^+))$.


Next, using the well known fact that $p(f)$ is recovered uniquely from $p(\seq_i(f^+))$ for all $i$, we can suppose that $f^+=\seq_1(f^+)$, i.e., that $f$ is integral.

In the case when $f^+$ is regular, i.e. when $f^+(k)\ne f^+(l)$ for $k\ne l$, Joseph states~\cite[Section 3.3]{Jo1} that $p(f)$ equals the shape of the output of the standard Robinson-Schensted algorithm~\cite[5.1.4, proof~of~Theorem~A]{Knu} applied to the unique Weyl group element ${\rm w}$ such that ${\rm w}(f^+)$ is dominant. It is easy to check that this statement is equivalent to the claim that $p(f)=\RS(f)$ in this case.

In the case when $f^+$ is not regular, following Joseph~\cite[Section 2.1]{Jo3} we replace $f$ by any function $f'$ such that $(f')^+$ is regular and $f$ belongs to the upper closure $\widehat F_{f'}$ of a certain facette $F_{f'}$ containing $f'$~\cite[Section 2.1]{Jo3}. In our language this means that $f'$ and $f$ satisfy the following conditions:



\begin{gather}\mbox{if~}f^+(i)> f^+(j)\mbox{,~then~}(f')^+(i)>(f')^+(j)\mbox{~for~all~}i, j\le n;\label{Gfa1}\\
\mbox{if~}f^+(i)< f^+(j)\mbox{,~then~}(f')^+(i)<(f')^+(j)\mbox{~for~all~}i, j\le n;\\
\mbox{if~}f^+(i)=f^+(j)\mbox{~and~}i<j\mbox{,~then~}(f')^+(i)>(f')^+(j).\label{Gfa3}\end{gather}
Then, according to Joseph, $p(f)=p(f')$~\cite[Section 2.4]{Jo2}. A direct checking using (\ref{Gfa1})-(\ref{Gfa3}) and the above linear order $\triangleleft$ shows that in this case $p(f')=\RS(f)$.



\end{proof}

\begin{Ex} Let $f=(\sqrt 2-1, 5, 9, \sqrt 2+3, 5,\sqrt 2+4, 7, 7)\in\mathbb C^8$.

1) $f^+=(\sqrt 2-1, 4, 7, \sqrt2, 1, \sqrt 2-1, 1, 0)$.

2-3) $\seq_1(f^+)=(\sqrt 2-1, \sqrt2, \sqrt 2-1) ~(n_1=3)$, $\seq_2(f^+)=(4,
7, 1, 1, 0)~(n_2=5)$.

4) $\widetilde{\seq_1(f^+)}=\{(\sqrt 2-1)', \sqrt 2, (\sqrt 2-1)''\}, \widetilde{\seq_2(f^+)}=\{4, 7, 1', 1'', 0\}$.

5) Applying the Robinson-Schensted algorithm we have$$\begin{array}{ll}\widetilde{\seq_1(f^+)}\mapsto
\begin{array}{cc}
\hline\multicolumn{1}{|c|}{\sqrt 2}&\multicolumn{1}{|c|}{(\sqrt 2-1)''}\\
\hline \multicolumn{1}{|c|}{(\sqrt 2-1)'}&\\
\hhline{-~}
\end{array}\mapsto (2, 1)~~,
&
\hspace{20pt}\widetilde{\seq_2(f^+)}\mapsto
\begin{array}{ccc}\hline
\multicolumn{1}{|c|}{7}&\multicolumn{1}{|c|}{1''}&\multicolumn{1}{|c|}{0}\\
\hline\multicolumn{1}{|c|}{4}&\multicolumn{1}{|c|}{1'}&\\
\hhline{--~}\end{array}\mapsto (3, 2)\end{array}.
$$

6) $p(\sqrt 2-1, 5, 9, \sqrt 2+3, 5,\sqrt 2+4, 7, 7)=(2, 1)\cup(3, 2)=(3, 2, 2, 1).$
\end{Ex}

\subsection{Rank of a partition}\label{SSSrrsa} Let, as above, $\EuScript
O(f)\subset\frak{sl}(n)^*$ be the nilpotent coadjoint orbit of $\frak{sl}(n)$ assigned to a function $f\in\mathbb C^n$. For $x\in\EuScript O(f)$, the rank of $x$ is independent on $x$ and equals $n-p(f)_{max}$, where $p(f)_{max}$ is the maximal element of the partition $p(f)$. By definition, the integer $p(f)_{max}$ is the {\it corank} of $p$.

\begin{lemma}\label{Lfcrk}Let $f\in\mathbb C^n$. The corank of $p(f)$ equals the length of a longest strictly decreasing subsequence of $f^+$ such that the difference between any two elements is an integer.\end{lemma}
\begin{proof} It is obvious that the
corank of $p(f)$ equals the maximum of coranks of $p_1,..., p_t$, where $p_1,..., p_t$ are the partitions defined in Step 5) of
Section~\ref{SSrsa}. It is known that for each $i$ the corank of $p_i=p(\widetilde{\seq_i(f^+)})$
equals to the length of a longest strictly decreasing
subsequence~\cite[p. 69, Ex. 7]{Knu} of $\widetilde{\seq_i(f^+)}$. For some $i_0$ a longest strictly decreasing subsequence of $\widetilde{\seq_{i_0}(f^+)}$ will also be a longest strictly decreasing subsequence of $f^+$ such that the difference between any two elements is an integer, and the lemma is proved.

\end{proof}

\subsection{Proof of
Proposition~\ref{Lfdcf}}\label{SSicf} Proposition~\ref{Lfdcf} is implied by the
following two lemmas.
\begin{lemma}\label{Lprcdm}Let $f\in \mathbb C^{\mathbb Z_{>0}}$. If $I(f)\ne0$, there exists $r\in\mathbb Z_{\ge0}$
such that any finite subset $F\subset \mathbb Z_{>0}$ has a subset
$F'\subset F$ so that $f|_{F'}$ is
integral and $|F\backslash F'|\le r$.
\end{lemma}
\begin{lemma}\label{Ldmldm}Fix $r\in\mathbb Z_{\ge0}$. If for any
finite subset $F\subset\mathbb Z_{>0}$ there is $F'\subset F$ so that $f|_{F'}$ is integral and
$|F\backslash F'|\le r$, then there
is a finite subset $F\subset \mathbb Z_{>0}$ such that
$f|_{\mathbb Z_{>0}\backslash F}$ is integral and $|F|\le
r$.\end{lemma}

\subsubsection{Proof of Lemma~\ref{Lprcdm}}\label{SSSc1} Due to the description
of the corank of $p(f)$ presented in Lemma~\ref{Lfcrk}, Lemma~\ref{Lprcdm} is
implied by the following lemma.
\begin{lemma}\label{Lrkf}Fix $f\in \mathbb C^{\mathbb Z_{>0}}$. If $I(f)\ne0$, then there exists $r\in\mathbb Z_{\ge0}$
such that $\APrk p(f|_{F})\le r$ for any finite subset $F\subset \mathbb Z_{>0}$.\end{lemma}
\begin{proof}Assuming that $I(f)\ne0$, pick $r$ as in Theorem~\ref{Tuslrk}.
Let $F$ be a finite subset of $\mathbb Z_{>0}$.

There is a nonzero homomorphism of $\frak{sl}(F)$-modules $\Ver(f|_{F})\to
L(f)$. Therefore, as $L(f|_{F})$ is the unique simple quotient of $\Ver(f|_{F})$,
$L(f|_{F})$ is isomorphic to a subquotient of $L(f)$ considered as
an $\frak{sl}(F)$-module. This implies $$(\Ue(\frak{sl}(F))\cap
I(f))\cdot L(f|_F)=0$$and$$
\Var(I(f)\cap\Ue(\frak{sl}(F)))\subset\frak{sl}(F)^{\le r}.$$ As all elements of $\Var(I(f|_F))$ are
nilpotent, we have $\APrk \EuScript O(f|_F)\le r$, and thus
$\APrk p(f|_{F})\le r$.\end{proof}

\subsubsection{Proof of Lemma~\ref{Ldmldm}}\label{SSSc2}We reduce the problem to a statement concerning the graph
$\Gamma:=(\mathbb Z_{>0}, E_f)$ attached to the pair $(\mathbb Z_{>0}, f)$  in the following way: the vertices of $\Gamma$ are the elements of $\mathbb Z_{>0}$, $E_f$ stands for the edges of $\Gamma$, and $i, j\in \mathbb Z_{>0}$ are connected by an edge if and only if
$f(i)-f(j)\notin\mathbb Z$.

Lemma~\ref{Ldmldm} is implied by the following lemma.
\begin{lemma} Let $\Gamma=(S, E)$ be a graph. Assume that there is $r\in\mathbb Z_{\ge0}$ so that any finite subset $F\subset
S$ decomposes into two subsets
\begin{center}$inf(F)\cup fin(F)$\end{center} with the properties
\begin{equation}a)~\Gamma|_{inf(F)}~\mbox{has no edges},\hspace{10pt} b)~|fin(F)|\le r.\label{Cgrf}\end{equation}
Then $S$ decomposes into two subsets
$inf(S)\cup fin(S)$ satisfying~(\ref{Cgrf}) with $F$ replaced by~$S$.\end{lemma}
\begin{proof} In what follows we say that a vertex of $S$ is connected with another vertex if they belong to a common edge. Denote by $S^{>r}$ the set of vertices of $S$ which belong to at least $r+1$ edges. Respectively, let $S^{\le r}$
be the set of vertices of $S$ which belong to at most $r$ edges. In addition,
denote by $\underline S^{\le r}$ the subset of $S^{\le r}$ consisting of vertices connected with at least one vertex from $S^{\le r}$.

We claim that both $S^{>r}$ and $\underline S^{\le r}$ are finite
and\begin{equation}\apro{1})~|S^{>r}|\le r,\hspace{10pt}\apro{2})~
|\underline S^{\le r}|\le r^2.\label{Etw}\end{equation}

First we show~(\ref{Etw}) under the assumption that $S^{>r}$ and $\underline S^{\le r}$ are finite. Let $\widetilde S^{>r}$ be a finite subset of
$S$ such that

1) $S^{>r}\subset \widetilde S^{>r}$,

2) any vertex from $S^{>r}$ is connected with at least $r+1$ vertices
form $\widetilde S^{>r}$ (such a subset $\widetilde S^{>r}$
always exists).\\
A vertex $i\in inf(\widetilde S^{> r})$ can be connected only
with vertices from $fin(\widetilde S^{>r})$, and hence $i\in S^{<r}$
by~(\ref{Cgrf})a). Therefore,

$$inf(\widetilde S^{>r})\subset S^{\le r}\cap\widetilde S^{>r}.$$This implies
\begin{equation}S^{>r}\subset
fin(\widetilde S^{>r})\label{Ewild},\end{equation} and since $|fin(\widetilde S^{>r})|<r$ by~(\ref{Cgrf})b), we obtain~(\ref{Etw})\apro{1}).

To prove (\ref{Etw})\apro{2}), note that since any vertex of $fin(\underline S^{\le r})$ belongs to at most $r$ edges, the entire set $fin(\underline S^{\le r})$ belongs to at most $r^2$ edges. As any vertex from $\underline S^{\le r}$ is connected with a vertex from $fin(\underline S^{\le r})$, we obtain (\ref{Etw})\apro{2}).

Now we drop the assumption that both $S^{>r}$ and
$\underline S^{\le r}$ are finite. Applying the preceding arguments
we show that~(\ref{Etw}) holds if we replace $S^{>r}$ and
$\underline S^{\le r}$ by their intersections with any finite subset of $S$.
Thus~(\ref{Etw}) holds also for $S^{>r}$ and $\underline S^{\le r}$.

To finish the proof, we set
$$fin(S):=fin(\widetilde S^{>r}\cup\underline S^{\le r}).$$
Then $|fin(S)|\le r$ by~(\ref{Cgrf})b). The same arguments by which  we prove~(\ref{Ewild}) imply$$S^{>r}\subset
fin(\widetilde S^{>r}\cup\underline S^{\le r}):=S.$$ Due to the
definition of $\underline S^{\le r}$, any vertex from $$S\backslash
(\widetilde S^{>r}\cup\underline S^{\le r})$$ can be connected only
with vertices from $S^{>r}$. Thus $\Gamma|_{S\backslash fin(S)}$ has no
edges, and the proof is complete.
\end{proof}

\subsection{Proof of Proposition~\ref{Lflc}}\label{SSflc}
\begin{lemma}\label{Lcex} Fix $r\in\mathbb Z_{\ge0}$. Let $f\in \mathbb C^{2r+2}$
be an integer valued function such that\begin{equation}f(2i)>
f(2i-1)\label{Enstr}\end{equation}for $1\le i\le r+1$. Then $\APrk
p(f)> r$.\end{lemma}
\begin{proof} Assume $\APrk p(f)\le r$. Then the sequence $f^+=(f(1), f(2)-1,..., f(n)-n+1)$
contains a strictly decreasing subsequence $\seq'$ of length at least $r+2$. The set $\{1,..., 2r+2\}$ is the disjoint union of $r+1$ pairs of the form $\{2i, 2i-1\}$, hence for some $i$
both $f(2i-1)-(2i-1)+1$ and $f(2i)-2i+1$ belong to $\seq'$. On the other hand,
$$f(2i-1)-(2i-1)+1\le f(2i)-2i+1$$ by~(\ref{Enstr}), thus $\seq'$ is not
strictly decreasing. This contradiction shows that
$\APrk p(f)> r$.\end{proof}
\begin{proof}[Proof of Proposition~\ref{Lflc}] Assume that $I(f)\ne0$ and pick
$r$ as in Lemma~\ref{Lrkf}. Using Proposition~\ref{Lftfmv} and Proposition~\ref{Lfdcf}, we reduce Proposition~\ref{Lflc} to the following statement:

If an integer valued function $f\in\mathbb C^{\mathbb Z_{>0}}$ takes finitely many values and there exists $r\in\mathbb Z_{\ge0}$ such that $\APrk p(f|_F)\le r$ for any finite subset $F\subset \mathbb Z_{>0}$,  then $f$ is
locally constant.

We prove this statement by induction on $|f|$. The base of induction $(|f|=1)$ is trivial.


Assume that the statement holds for $|f|=n\ge1$, and let $f$ be a function which
takes precisely $(n+1)$ values. Let $M$ be the maximal value of $f$. Say that $i, j\in \mathbb Z_{>0}, i\ne j,$ are equivalent whenever one of the following conditions hold:

1) $i\prec j$, $f(i)=f(j)=M$, and $f(s)=M$, for any $s$, $i\prec s\prec
j$;

2) $i\prec j$, $f(i)<M$, $f(j)<M$, and $f(s)<M$, for any $s$, $i\prec
s\prec j$.\\
It is easy to see that this this is a well defined equivalence relation on $\mathbb Z_{>0}$. There are two possibilities for the respective equivalence classes $S_\alpha$:

a) $f(s)=M$ for any $s\in S_\alpha$;

b) $f(s)<M$ for any $s\in S_\alpha$.

We claim that there exist no more than $r+1$ equivalence classes of type b).
Assume to the contrary that $s_0\prec s_2\prec...\prec s_{2r+2}$ are elements from $r+2$ distinct equivalence classes of type b). Then, for any
$i,~0\le i\le r$, there exists $s_{2i+1}\in S$ such that
$$f(s_{2i+1})=M\mbox{~and~}s_{2i}\prec s_{2i+1}\prec s_{2i+2}.$$ The restriction of $f$ to the set $F:=\{s_0, s_1,..., s_{2r+2}\}$ satisfies the assumption of Lemma~\ref{Lcex}. Hence
$\APrk p(f|_F)> r$, which contradicts the statement of Lemma~\ref{Lrkf}.

Therefore, there are at most $r+1$ equivalence classes $S_\alpha$ of type b). Any two classes of type a) must be
separated by a class of type b), and hence there are at most $r+2$ equivalence
classes of type a). In particular the partition $\sqcup_\alpha S_\alpha=\mathbb Z_{>0}$ is finite.

Clearly, $f$ takes at most $n$ values on each $S_\alpha$. By the induction assumption each $S_\alpha$ admits a compatible partition such that $f|_{S_\alpha}$ is locally constant. Therefore, $f$ is also locally constant.\end{proof}
\section{Proof of Theorem~\ref{1T}~b)}\label{Spr1tb}
Theorem~\ref{1T}b) is a corollary of the following result.
\begin{proposition}\label{Pt3+} Let $f\in\mathbb C^{\mathbb Z_{>0}}$ be a locally constant and almost integral function. Then there is a nonzero integrable ideal $I$ of $\Ue(\frak{sl}(\infty))$ such that $I\subset I(f)$.\end{proposition}

We will prove a more precise version of this result. Let $S_1\sqcup...\sqcup S_t=\mathbb Z_{>0}$ be a fixed finite partition of $\mathbb Z_{>0}$ compatible with the order $\prec$. Denote by $S_{i_1},..., S_{i_x}$ all infinite sets in this partition. By $\gamma$ we denote the total number of elements in the finite sets of the partition. Let $f\in\mathbb C^{\mathbb Z_{>0}}$ be a function locally constant with respect to the partition $S_1\sqcup...\sqcup S_t$. It is easy to see that $f\in\mathbb C^{\mathbb Z_{>0}}$ is almost integral if and only if $f(j)-f(k)\in\mathbb Z$ for any $j\in S_{j'}$ and $k\in S_{k'}$ such that both $S_{j'}$ and $S_{k'}$ are infinite. Under the assumption that $f$ is almost integral, we set
\begin{equation}\nint(f):=\sum\limits_{1\le j< x}\max(0, f(S_{i_{j+1}})-f(S_{i_j})),\hspace{10pt} \wid(f):=\sum\limits_{1\le j< x}\max(f(S_{i_{j}})-f(S_{i_{j+1}}),~0),\label{Enint}\end{equation}where $f(S_i)$ is the value of $f$ on any element of $S_i$ (we recall that $f$ is constant on $S_i$).

The following proposition is a more precise version of Proposition~\ref{Pt3+} and compares the annihilator of a simple highest weight module with the annihilator of a c.l.s. We will prove it by first establishing a finite-dimensional analogue, namely Proposition~\ref{Pailf-f}, and then showing that Proposition~\ref{Pailf} actually reduces to this finite-dimensional analogue.
\begin{proposition}\label{Pailf} Let $f\in\mathbb C^{\mathbb Z_{>0}}$ be a function, locally constant with respect to the partition $S_1\sqcup...\sqcup S_t$ of $\mathbb Z_{>0}$. Then $$I(\mathcal L_{(\nint(f)+\gamma)}^\infty \mathcal E^{\wid(f)})\subset I(f).$$\end{proposition}
Let $F$ be a finite subset of $\mathbb Z_{>0}$. Clearly, $$(S_1\cap F)\sqcup...\sqcup(S_t\cap F)$$ is a partition of $F$.
We wish to define $\nint(f')$ and $\wid(f')$ by formulas analogous to~(\ref{Enint}) for any function $f'\in\mathbb C^F$ which is locally constant with respect to the partition $(S_1\cap F)\sqcup...\sqcup (S_t\cap F)$.
For this purpose we denote by $S_1'$ the first $S_{i_j}$ for which $S_{i_j}\cap F\ne\emptyset$, by $S_2'$ the second $S_{i_j}$ for which $S_{i_j}\cap F\ne\emptyset$ and so on. Then we define $\nint(f')$ and $\wid(f')$ by the respective right-hand sides of~(\ref{Enint}) applied to the subsets $(S_1'\cap F), (S_2'\cap F),...$ instead of $S_{i_1}, S_{i_2},...$. Finally, $\gamma(F)$ stands for the total number of elements in all intersections $S_i\cap F$ for finite sets $S_i$. For a large enough $F$ we have $\gamma(F)=\gamma, \wid(f|_F)=\wid(f), \nint(f|_F)=\nint(f)$.
\begin{proposition}\label{Pailf-f}Let $F\subset\mathbb Z_{>0}$ be a finite subset with $n$ elements, and $f'\in\mathbb C^F$ be a function locally constant with respect to the partition $(S_1\cap F)\sqcup...\sqcup (S_t\cap F)=F$. Then $$I((\mathcal L_{(\nint(f')+\gamma(F))}^\infty \mathcal E^{\wid(f')})_n)\subset I(f').$$\end{proposition}

For the proof of Proposition~\ref{Pailf-f} we need two lemmas (Lemmas~\ref{Lre} and~\ref{Lailf-f} below) and some more notation. In Lemma~\ref{Lre} $f=(f_1, ..., f_n)$ stands for a function $f\in\mathbb C^n$. We set $L(f_1,..., f_n):=L(f)$ and $I(f_1,..., f_n):=I(f)$ (where the fixed order on $\{1, 2,..., n\}$ is the standard one). For a fixed nonnegative integer $s< n$ and $z_0\in\mathbb C$, we put:  $$\tilde f=(f_1,..., f_s, z_0, f_{s+1},..., f_n)\in\mathbb C^{n+1}.$$
If $A, B$ are two subsets of $\Irr_n$, $A\otimes B$ stands for the set of isomorphism classes of all simple constituents of the tensor products $\alpha\otimes \beta$ for $\alpha\in A$ and $\beta\in B$.
\begin{lemma}\label{Lre}Let $Q_n$ be a subset of $\Irr_n$ such that $$ I(Q_n)\subset I(f_1,..., f_s, f_{s+1}-1,..., f_n-1).$$ Then $$I((\mathcal L_1^\infty)_n\otimes Q_n)\subset I(\tilde f),$$ $I(Q_n\otimes (\mathcal L_1^\infty)_n)$ being an ideal of $\Ue(\frak{sl}(n))$ and $I(\tilde f)$ being an ideal of $\Ue(\frak{sl}(n+1))$.\end{lemma}
\begin{proof}Our idea is to replace $z_0$ by a ``generic value''. To do this, consider the supplementary Lie algebras
\begin{center} $\frak{sl}(n+1)[z]:=\frak{sl}(n+1)\otimes_{\mathbb C}\mathbb C[z] \subset\frak{sl}(n+1)(z):=\frak{sl}(n+1)\otimes_{\mathbb C}\overline{\mathbb C(z)},$\end{center}the larger Lie algebra $\frak{sl}(n+1)(z)$ being finite dimensional and simple over the algebraically closed field $\overline{\mathbb C(z)}$. The sequence $\hat f:=(f_1,..., f_s, z, f_{s+1},..., f_n)$ of elements of $\overline{\mathbb C(z)}$ defines a weight $\lambda_{\hat f}\in\frak h_{n+1}^*\otimes\overline{\mathbb C(z)}$.

Applying the equality~(\ref{Eww0}) to $\hat f$, we obtain
$$I(\hat f)=I(f_1,...f_s, f_{s+1}-1, f_{s+2}-1,..., f_n-1, z+n-s).$$By Proposition~\ref{Lpind}, we have
$$L(f_1,...f_s, f_{s+1}-1, f_{s+2}-1,..., f_n-1, z+n-s)\cong \Ue(\frak{sl}(n+1)(z))\otimes_{\Ue(\frak p)}L(f_1,...f_s, f_{s+1}-1, f_{s+2}-1,..., f_n-1, z+n-s)^{\frak n},$$where  $\frak p$ is a parabolic subalgebra of $\frak{sl}(n+1)(z)$ with a semisimple part $\frak{sl}(n)(z)$ and nilradical $\frak n$. Proposition~\ref{Lpind} yields also an isomorphism of $\frak{sl}(n)(z)$-modules
$$L(f_1,...f_s, f_{s+1}-1, f_{s+2}-1,..., f_n-1, z+n-s)^{\frak n}\cong L(f_1,...f_s, f_{s+1}-1, f_{s+2}-1,..., f_n-1)\otimes_{\mathbb C}\overline{\mathbb C(z)}.$$


Therefore we have an isomorphism of $\frak{sl}(n)$-modules $$L(f_1,...f_s, f_{s+1}-1, f_{s+2}-1,..., f_n-1, z+n-s)\cong L(f_1,...f_s, f_{s+1}-1, f_{s+2}-1,..., f_n-1)\otimes_{\mathbb C}\Sa^\cdot(\overline{\mathbb C(z)}^{\!\!~n}).$$ Hence $L(f_1,...f_s, f_{s+1}-1, f_{s+2}-1,..., f_n-1, z+n-s)$ is annihilated by $I(Q_n\otimes (\mathcal L_1^\infty)_n)$, i.e., $$I(Q_n\otimes (\mathcal L_1^\infty)_n)\subset I(f_1,...,f_s, f_{s+1}-1, f_{s+2}-2,..., f_n-1, z+n-s)=I(\hat f).$$
For this reason it suffices to show that $$I(\hat f~)\cap\Ue(\frak{sl}(n+1))\subset I(\tilde f)$$for any $z_0\in\mathbb C$.

Let $v_{\hat f}$ be a highest weight vector of the $\frak{sl}(n+1)(z)$-module $L(\hat f)$. 
Consider the  $\Ue(\frak{sl}(n+1)[z])$-module\begin{equation}\Ue(\frak{sl}(n+1)[z])\cdot v_{\hat f}.\label{Et}\end{equation} Clearly, the action of $\frak h_{n+1}$ on~(\ref{Et}) is semisimple. The $\lambda_{\hat f}$-weight space of~(\ref{Et}) coincides with $\Ue(\frak h_{n+1}\otimes\mathbb C[z])\cdot v_{\hat f},$ and is isomorphic to $\mathbb C[z]$ as a $\mathbb C[z]$-module. Therefore, the $\lambda_{\hat f}$-weight space of the quotient \begin{equation}\Ue(\frak{sl}(n+1)[z])\cdot v_{\hat f}/(z-z_0)\Ue(\frak{sl}(n+1)[z])\cdot v_{\hat f}\label{Etq}\end{equation} is one-dimensional. In particular, the quotient~(\ref{Etq}) is nonzero.

Obviously,~(\ref{Etq}) is annihilated by \begin{equation}I(\hat f)\cap\Ue(\frak{sl}(n+1)).\label{Etqi}\end{equation} On the other hand,~(\ref{Etq}) has a highest weight vector of weight $\lambda_{\tilde f}$, and thus $L(\tilde f)$ is annihilated by~(\ref{Etqi}). This is precisely what we have to prove.
\end{proof}

\begin{lemma}\label{Lailf-f} Let $F$ be a finite subset of $\mathbb Z_{>0}$ with $n$ elements, and $f'\in\mathbb C^F$ be a function locally constant with respect to the partition $(S_1\cap F)\sqcup...\sqcup(S_t\cap F)=F$.
After identification of $\frak{sl}(F)$ with $\frak{sl}(n)$ we have
$$I((\mathcal L_{(\nint(f')+\gamma(F))}^\infty \mathcal E^{\wid(f')})_{n-(\gamma(F)+\nint(f'))})
\subset I(f'),$$ where $I((\mathcal L_{(\nint(f')+\gamma(F))}^\infty \mathcal E^{\wid(f')})_{n-(\gamma(F)+\nint(f'))})\subset \Ue(\frak{sl}(n-(\gamma(F)+\nint(f')))$ and $I(f')\subset\Ue(\frak{sl}(n))$.
\end{lemma}
\begin{proof} We prove this lemma by induction on $\gamma(F)+\nint(f')$.



Let $\gamma(F)=\nint(f')=0$. Then $f'$ is dominant, $L(f')$ is integrable and the statement of Lemma~\ref{Lre} follows from Lemma~\ref{Lexpl}.

Next, assume that $\gamma(F)+\nint(f')=k+1$ and that our statement holds for $\gamma(F)+\nint(f')\le k$. Then $\nint(f')>0$ or $\gamma(F)>0$. We consider both possibilities.

Let $\nint(f')>0$. Then $f'(S_j')<f'(S_{j+1}')$ for some $j$. Denote by $s$ the maximal element of $S_j'\cap F$ (with respect to the order inherited from the order $\prec$). Put $$F_-:=F\backslash s,\hspace{10pt}f'_-:=f'|_{F_-}\in\mathbb C^{F_-},$$and note that $f'_-$ is locally constant with respect to the partition $$(S_1\cap F_-)\sqcup...\sqcup(S_t\cap F_-)$$ of $F_-$. Moreover, it is easy to see that
\begin{center}$\gamma(F_-)=\gamma(F),~\nint(f'_-)<\nint(f')$ and $\wid(f'_-)\le\wid(f')$.\end{center} Thus we can apply the induction assumption to $f'_-$, which yields
$$I((\mathcal L_{\gamma(F)+\nint(f'_-)}^\infty\otimes\mathcal E^{\wid(f'_-)})_{n-1-(\gamma(F)+\nint(f'_-))})\subset I(f'_-).$$
Applying Lemma~\ref{Lre} to $s, z_0=f'(s),$ we obtain
\begin{equation}I((\mathcal L_{\gamma(F)+\nint(f'_-)+1}^\infty\mathcal E^{\wid(f'_-)})_{n-1-(\gamma(F)+\nint(f'_-))})\subset I(f').\label{Eh}\end{equation}
Since\begin{center}$\gamma(F_-)+\nint(f'_-)+1\le \gamma(F)+\nint(f')$ and $n-1-(\gamma(F_-)+\nint(f'_-))\ge n-(\gamma(F)+\nint(f'))$,\end{center}(\ref{Eh}) implies
$$I((\mathcal L_{\gamma(F)+\nint(f')}^\infty\mathcal E^{\wid(f')})_{n-(\gamma(F)+\nint(f'))})\subset I(f'),$$
which is precisely what we need to prove.

In the case when $\nint(f')=0, \gamma(F)>0$ we pick $s$ to be the least element of $F\backslash\cup_{j\le x}S_{i_j}$ with respect to the order inherited from $\prec$. Then we apply the same arguments as above.
\end{proof}
\begin{remark}\label{R1} It is clear that Lemma~\ref{Lailf-f} applies to an arbitrary linearly ordered finite set $F$, an arbitrary compatible partition of $F$, an arbitrary function $f\in\mathbb C^F$ locally constant with respect to this partition, and an arbitrary choice of equivalence classes of this partition used to define $\nint(\cdot), \wid(\cdot)$ and $\gamma(\cdot)$.\end{remark}

\begin{proof}[Proof of Proposition~\ref{Pailf-f}] Identify $F$ with $\{1,.., n\}$ as ordered sets (the order on $F$ being inherited from the order $\prec$). The function $f'\in\mathbb C^F$ becomes $f'=(f_1',..., f_n')\in\mathbb C^n$. Let $s$ be the least element of $S_1'\cap F$ under the above identification. Put $$\check f:=(f_1',...,\hspace{-10pt}\underbrace{f_s', f_s',..., f_s'}_{(\gamma(F)+\nint(f')+1)-~\mathrm{times}},\hspace{-15pt}, f_{s+1}',..., f_n').$$It is clear that $\check f$ is locally constant with respect to the partition $\check S_1\sqcup\check S_2\sqcup...=\{1,..., n+\nint(f')+\gamma(F)\}$, which is defined as follows:

(1) $\check S_i$ coincides with $(S_i\cap F)$ for $i<j$, where $j$ is defined by the equality $S_j=S_1'$;

(2) $\check S_j=(S_j\cap F)\cup\{\check s+1,..., \check s+\nint(f')+\gamma(F)\},$ where $\check s$ is the image in $\{1,..., n\}$ of the last element of $S_j\cap F$;

(3) $\check S_i=\{\check s_-+\gamma(F)+\nint(f'), \check s_-+\gamma(F)+\nint(f')+1,..., \check s_++\gamma(F)+\nint(f')-1, \check s_++\gamma(F)+\nint(f')\}$ for $i>j$, where $\check s_-$ and $\check s_+$ are is the images in $\{1,..., n\}$ of the least and the greatest elements of $S_i\cap F$.

Remark~\ref{R1} enables us to apply Lemma~\ref{Lailf-f} to the function $\check f$ and the partition $\check S_1\sqcup\check S_2\sqcup...=\{1,..., n+\nint(f')+\gamma(F)\}$:
$$I((\mathcal L_{\gamma(F)+\nint(f')}^\infty\mathcal E^{\wid(f'))})_{n})\subset I(\check f).$$

Finally, since $L(f)$ is an $\frak{sl}(n)$-subquotient of $L(\check f)$, we have $I(\check f)\cap\Ue(\frak{sl}(n))\subset I(f')$, and Proposition~\ref{Pailf-f} is proved.\end{proof}
Proposition~\ref{Pailf} follows now from Proposition~\ref{Pailf-f} and the next lemma.
\begin{lemma}\label{Lidmo}Let $I\subset\Ue(\frak{sl}(\infty))$ be an ideal, and $f\in\mathbb C^{\mathbb Z_{>0}}$ be a function. Then $I\subset I(f)$ if and only if $I_F:=I\cap\Ue(\frak{sl}(F))$ annihilates $L(f|_F)$ for any finite subset $F$ of $\mathbb Z_{>0}$.\end{lemma}
\begin{proof}
Let $I\subset I(f)$. Denote by $v_f$ a highest weight vector of $L(f)$. If $F$ is a finite set, then $\Ue(\frak{sl}(F))\cdot v_f$ is a highest weight $\frak{sl}(F)$-submodule of $L(f)$. Thus $L(f|_F)$ is isomorphic to a subquotient of  $L(f)$, and consequently $I_F=I\cap\Ue(\frak{sl}(F))$ annihilates $L(f|_F)$.

We now prove the converse. Set $$\Ver(F):=\Ver(f|_F)/(I\cap\Ue(\frak{sl}(F))\cdot\Ver(f|_F).$$As $I\cap\Ue(\frak{sl}(F))$ annihilates $L(f|_F)$, $\Ver(F)$ is a nonzero highest weight $\frak{sl}(F)$-module. Let $v_f(F)$ be a highest weight vector of $M(F)$. For any finite subsets $F_1\subset F_2\subset\mathbb Z_{>0}$, there exists a unique morphism of $\frak{sl}(F_1)$-modules$$\phi_{F_1, F_2}: \Ver(F_1)\to\Ver(F_2)$$such that
$\phi_{F_1, F_2}(v_f(F_1))=v_f(F_2)$. This defines a direct
system of morphisms $$\{\phi_{F_1, F_2}\}_{F_1\subset F_2},$$ and we denote its limit by $\tilde\Ver(f)$.

By definition, $I$ annihilates the $\frak{sl}(\infty)$-module $\tilde\Ver(f)$. Our construction guarantees that $\tilde \Ver(f)$ contains a highest
vector $v_f:=\varinjlim v_{f|_{F_i}}$ of weight $\lambda_f$. Thus $L(f)$ is isomorphic to a
simple quotient of $\tilde \Ver(f)$, which implies $I\subset I(f)$.\end{proof}

\section{Proof of Theorem~\ref{2T} and Proposition~\ref{Pp}}\label{Spr2t}
Theorem~\ref{2T} is implied by the following propositions.

\begin{proposition}\label{Pint}Let $\frak b\supset\frak h$ be a splitting Borel subalgebra of $\frak{sl}(\infty)$, and $f\in\mathbb C^{\mathbb Z_{>0}}$ be function. Then $I=\Ann_{\Ue(\frak{sl}(\infty))} L_{\frak b}(f)$ is a prime integrable ideal of $\Ue(\frak{sl}(\infty))$.\end{proposition}
\begin{proposition}\label{Piint}Let $I$ be a prime integrable ideal of $\Ue(\frak{sl}(\infty))$ and $\frak b^0\supset\frak h$ be an ideal Borel subalgebra of $\frak{sl}(\infty)$. Then $I=\Ann_{\Ue(\frak{sl}(\infty))} L_{\frak b^0}(f^0)$ for some $f^0\in\mathbb C^{\mathbb Z_{>0}}$.\end{proposition}
\subsection{Proof of Proposition~\ref{Pint}}\label{SSint} 
The annihilator of a simple module is always prime, therefore in order to prove Proposition~\ref{Pint} we have to prove that the ideal $\Ann_{\Ue(\frak{sl}(\infty))} L_\frak b(f)$ is integrable for any $\frak b$ and any $f\in\mathbb C^{\mathbb Z_{>0}}$.
This is a direct consequence of the following three statements.

\begin{proposition}\label{Ppsisb}Let $S$ be an infinite subset of $\mathbb Z_{>0}$ and $\phi: \mathbb Z_{>0}\to S$ be a fixed bijection. Let $I$ be an ideal of $\Ue(\frak{sl}(\infty))$. Then the induced isomorphism $\phi: \Ue(\frak{sl}(\infty))\to \Ue(\frak{sl}(S))$ identifies $I$ and $I\cap\Ue(\frak{sl}(S))$.\end{proposition}
\begin{proof} Fix the exhaustion~(\ref{Eex}) and assume that $\frak{sl}(n)$ is generated by $e_{ij}$ for $i\ne j, i,j\le n$. Then $\frak{sl}(S)=\cup_m\frak{sl}(S_m)$, where $S_m$ is the image of $\{1,..., m\}$ under $\phi$. We have $$I\cap\Ue(\frak{sl}(S))=\cup_m(I\cap\Ue(\frak{sl}(S_m))).$$Since, for every $n\ge1$, $\frak{sl}(n)$ is $\Adj\frak{sl}(\infty)$-conjugate to $\frak{sl}(S_n)$, Lemma~\ref{Lconj} yields $$\phi^{-1}(I\cap\Ue(\frak{sl}(S)))=\phi^{-1}(\cup_n(I\cap\Ue(\frak{sl}(S_n))))=\cup_n(I\cap\Ue(\frak{sl}(n)))=I.$$
\end{proof}
\begin{corollary}\label{Psisb}Let $M$ be an $\frak{sl}(\infty)$-module and $S$ be an infinite subset of $\mathbb Z_{>0}$. Then $\Ann_{\Ue(\frak{sl}(\infty))}M$ is an integrable ideal in $\Ue(\frak{sl}(\infty))$ if and only $\Ann_{\Ue(\frak{sl}(S))}M$ is an integrable ideal of $\Ue(\frak{sl}(S))$.\end{corollary}

\begin{proposition}\label{Psai}Let $\frak b$ and $f$ be as in Proposition~\ref{Pint}. If $\Ann_{\Ue(\frak{sl}(\infty))} L_\frak b(f)\ne0$, then there exists an infinite subset $S\subset\mathbb Z_{>0}$
such that $L_\frak b(f)$ is an integrable
$\frak{sl}(S)$-module.\end{proposition}
\begin{proof}
As $\Ann_{\Ue(\frak{sl}(\infty))} L_\frak b(f)\ne0$, $f$ is locally constant relative to some partition $S_1\sqcup...\sqcup S_t=\mathbb Z_{>0}$, compatible with the order determined by $\frak b$. Since $S_i$ is infinite for some $i$, we apply Proposition~\ref{Ppint} to conclude that $L_\frak b(f)$ is an integrable $\frak{sl}(S_i)$-module.\end{proof}

\subsection{Proof of Proposition~\ref{Piint}}\label{SSiint}
Let $\frak b^0\supset\frak h$ be a fixed ideal Borel subalgebra of $\frak{sl}(\infty)$. The goal of this section is to show that any integrable ideal is an
annihilator of some $\frak b^0$-highest weight module of $\frak{sl}(\infty)$, and thus to prove Proposition~\ref{Piint}. Due to the fact that an arbitrary irreducible c.l.s. has the form $\mathcal L^\infty_lQ\mathcal R^\infty_r$ for $l,r\in\mathbb Z_{\ge0}$ and some irreducible c.l.s. of finite type $Q$, it is enough to prove the following proposition.

\begin{proposition}\label{Pclsid}For any
irreducible c.l.s. $Q$ of finite type and any $l, r\in\mathbb Z_{\ge0}$ there
exists $f\in \mathbb C^{\mathbb Z_{>0}}$ such that $$\Ann_{\Ue(\frak{sl}(\infty))}L_{\frak b^0}(f)=I(\mathcal L^\infty_lQ\mathcal R^\infty_r).$$\end{proposition}
We fix $l, r\in\mathbb Z_{\ge0}$. According to Proposition~\ref{Ppsisb}, the ideals $\Ann_{\Ue(\frak{sl}(\infty))}M$ and $\Ann_{\Ue(\frak{sl}(S))}M$ can be identified for any $\frak{sl}(\infty)$-module $M$ and any infinite subset $S$ of $\mathbb Z_{>0}$. Therefore, Proposition~\ref{Pclsid} is implied by 
the following lemma.
\begin{lemma}\label{Lclsidg}For any irreducible
c.l.s. $Q$ of finite type, there exist $f\in \mathbb C^{\mathbb Z_{>0}}$ and
an infinite subset $S\subset\mathbb Z_{>0}$ such that the $\frak{sl}(\infty)$-module $L_{\frak b^0}(f)$ is integrable as
an $\frak{sl}(S)$-module and the c.l.s. for $\frak{sl}(S)$ of $L_{\frak b^0}(f)$ equals to $\mathcal L^\infty_lQ\mathcal R^\infty_r$.\end{lemma}
We now prove Lemma~\ref{Lclsidg} by pointing out a concrete set $S$ for which the claim of the lemma holds. We recall that the ideal Borel subalgebra $\frak b^0$ defines a partition $S_1\sqcup S_2\sqcup S_3$ of $\mathbb Z_{>0}$. Let $F_l$ be the set consisting of the first $l$ elements of $S_1$. As an ordered set $F_l$ is isomorphic to $\{1,...., l\}$ with the standard order. Let $F_r$ be set consisting of the the last $r$ elements of $S_3$. As an ordered set $F_r$ is isomorphic to $\{-r,...., -1\}$ with the standard order. Put $$S:=\mathbb Z_{>0}\backslash (F_l\sqcup F_r).$$
Note that $\frak b^0_S:=\frak b^0\cap\frak{sl}(S)$ is an ideal Borel subalgebra of $\frak{sl}(S)$. Therefore, Proposition~\ref{CexY} c) asserts that, for any c.l.s. $Q$ of finite type, there is a $\frak b^0_S$-dominant function $f^0\in \mathbb C^S$ such that $Q=\cls(L_{\frak b^0_S}(f))$. For this reason Lemma~\ref{Lclsidg} is a direct corollary of the
following lemma.

\begin{lemma}\label{Lclspar}Let $f\in \mathbb C^{\mathbb Z_{>0}}$ satisfy the conditions

1) $f|_S\in \mathbb C^S$ is $\frak b^0_S$-dominant,

2) $f(i)-f(j)\notin\mathbb Z$ for any $i\ne j, j\in F_l$,

3) $f(i)-f(j)\notin\mathbb Z$ for any $i\ne j, j\in F_r$.\\
Then the c.l.s. of the $\frak{sl}(S)$-module $L_{\frak b^0}(f)$ is equal to
$\mathcal L^\infty_l\cls(L_{\frak b^0_S}(f|_S))\mathcal R^\infty_r$.\end{lemma}
\begin{proof}
By Proposition~\ref{Lpind},
$$L_{\frak b^0}(f)\cong \Ue(\frak{sl}(\infty))\otimes_{\Ue(\frak p)}L_{\frak b^0}(f)^{\frak n},$$where $L_{\frak b^0}(f)^{\frak n}\cong L_{\frak b^0_S}(f|_S)$ as an $\frak{sl}(S)$-module. Hence, there is an isomorphism of $\frak{sl}(S)$-modules
\begin{equation}L_{\frak b^0}(f)\cong \Sa^\cdot(\frak{sl}(\infty)/\frak p) \otimes_{\mathbb C} L_{\frak b^0_S}(f|_S).\label{Etpind}\end{equation}

Furthermore, there is an isomorphism of $\frak{sl}(S)$-modules \begin{equation}\frak{sl}(\infty)/\frak p\cong (V(S)\otimes\mathbb C^l\oplus\mathbb C^{\frac{l(l-1)}2})\oplus(V(S)_*\otimes\mathbb C^r\oplus\mathbb C^{\frac{r(r-1)}2})\oplus\mathbb C^{rl},\label{Eraad}\end{equation}where $V(S)$ is the natural $\frak{sl}(S)$-module and $\mathbb C$ stands for the one-dimensional trivial $\frak{sl}(S)$-module. Thus,\begin{equation}\Sa^\cdot(\frak{sl}(\infty)/\frak p)\cong
\Sa^\cdot(V(S)\otimes\mathbb C^l)\otimes\Sa^\cdot(V(S)_*\otimes\mathbb C^r)\otimes\Sa^\cdot(\mathbb C^{\frac{l(l-1)}2}\oplus\mathbb C^{\frac{r(r-1)}2})\oplus\mathbb C^{rl}).\end{equation} The c.l.s. of
$\Sa^\cdot(V(S)\otimes\mathbb C^l)=\Sa^\cdot(V(S))^{\otimes l}$ coincides with $\mathcal L^\infty_l$,  and the c.l.s. of
$\Sa^\cdot(V(S)_*\otimes\mathbb C^r)=\Sa^\cdot(V(S)_*)^{\otimes r}$ coincides with $\mathcal R^\infty_r$. Hence the c.l.s. of $\Sa^\cdot(\frak{sl}(\infty)/\frak p)$ as an $\frak{sl}(S)$-module coincides with $\mathcal
L_l^\infty\mathcal R^\infty_r$, and the proof is complete.\end{proof}
\begin{Ex} Consider the fixed exhaustion~(\ref{Eex}) of $\frak{sl}(\infty)$. Note that there is a canonical injection of $\frak{sl}(i)$-modules $\Sa^iV_i\to \Sa^{i+1}V_{i+1}$, where $V_i$ and $V_{i+1}$ are respectively the natural representations of $\frak{sl}(i)$ and $\frak{sl}(i+1)$. The direct limit $D:=\varinjlim_i \Sa^iV_i$ is a simple integrable $\frak{sl}(\infty)$-module which is multiplicity free as an $\frak h$-module. The module $D$ has no highest weight with respect to any splitting Borel subalgebra $\frak b$~\cite{DP}. The c.l.s. corresponding to $D$ equals $\mathcal L_1^{\infty}$, and in particular has infinite type. Lemma~\ref{Lclspar} implies that $\Ann_{\Ue(\frak{sl}(\infty))}D$ equals to the annihilator of a simple nonintegrable highest weight module. Indeed, let $\frak b^0$ be the ideal Borel subalgebra corresponding to the order (\apro{3}) in Section~\ref{SSslnot} and let $f$ be the function $$f(1)=\alpha\notin\mathbb Z, \hspace{10pt}f(n)=0,\hspace{10pt} n>1.$$Then $\Ann_{\Ue(\frak{sl}(\infty))}D=\Ann_{\Ue(\frak{sl}(\infty))} L_{\frak b^0}(f)$. This example illustrates the role of simple integrable non-highest weight modules in Theorem~\ref{2T}: the annihilators of such simple modules arise as annihilators of simple nonintegrable highest weight modules.\end{Ex}

\subsection{Proof of Proposition~\ref{Pp}} It remains to prove Proposition~\ref{Pp}.
\begin{proof}[Proof of Proposition~\ref{Pp}]
We say that an ideal $I$ of $\Ue(\frak{sl}(\infty))$ is {\it of locally finite codimension} if $I\cap\Ue(\frak g)$ has finite codimension in $\Ue(\frak g)$ for any finite-dimensional subalgebra $\frak g\subset\frak{sl}(\infty)$. It is easy to see that such ideals have the following remarkable properties:

(\apro{1}) the map $Q\mapsto I(Q)$ identifies the set of c.l.s. of finite type with the set of ideals of locally finite codimension;

(\apro{2}) if an $\frak{sl}(\infty)$-module $M$ is annihilated by an ideal $I\subset\Ue(\frak{sl}(\infty))$ of locally finite codimension, then $M$ is integrable.

Using the properties (\apro{1}) and (\apro{2}) one observes that if $\frak b$ is a Borel subalgebra, such that for any prime ideal $I$ of locally finite codimension there exists $f\in\mathbb C^{\mathbb Z_{>0}}$ with $I=\Ann_{\Ue(\frak{sl}(\infty))}L_\frak b(f)$, then $\frak b$ is ideal. Indeed, Proposition~\ref{Lexpl}a) gives an explicit expression of $\cls(L_\frak b(f))$ in terms of $f$. The requirement that this procedure allows for every c.l.s. of finite type to appear in the right-hand side of~(\ref{Eclsft}) forces the existence of a $\prec$-compatible decomposition $\mathbb Z_{>0}=F\sqcup S\sqcup F'$, where $F$ and $F'$ are arbitrary finite sets. Clearly, this is equivalent to the requirement that $\frak b$ is ideal.

\end{proof}
\section{On simple $\frak{sl}(\infty)$-modules determined up to isomorphism by their annihilators}\label{Snew}
It is a remarkable fact that if $\frak g$ is finite dimensional and semisimple, then a simple $\frak g$-module $M$ is determined up to isomorphism by its annihilator in $\Ue(\frak g)$ if and only if $M$ is finite dimensional. We now provide an analogue of this fact for $\frak{sl}(\infty)$.

Recall that a {\it simple tensor module} of $\frak{sl}(\infty)$ is a simple submodule of the tensor algebra \begin{center}T$^\cdot(V(\infty)\oplus V(\infty)_*)$\end{center}~\cite{DPS, PS}. It is easy to check that, for any fixed ideal Borel subalgebra $\frak b^0$, the simple tensor modules are precisely the highest weight modules $L_{\frak b^0}(f)$ such that $f$ can be chosen to be 0 almost everywhere (recall that the isomorphism class of a module $L_{\frak b^0}(f)$ recovers $f$ up to an additive constant).
\begin{proposition} Let $M$ be a simple $\frak{sl}(\infty)$-module which is determined up to isomorphism by its annihilator $I=\Ann_{\Ue(\frak{sl}(\infty))}M$. If $I$ is integrable, then $M$ is isomorphic to a simple tensor module.\end{proposition}
\begin{proof} 
If $I$ is not of locally finite codimension, then Lemma~\ref{Lclspar} shows that our assumption on $M$ is contradictory as the function $f$  from Lemma~\ref{Lclspar} is not determined uniquely by $I$ up to an additive constant by $I$. In other words, if $I$ is not of locally finite codimension, then Lemma~\ref{Lclspar} implies that there exist $f_1, f_2\in \mathbb C^{\mathbb Z_{>0}}$ such that $$\Ann_{\Ue(\frak{sl}(\infty))} L_{\frak b^0}(f_1)=\Ann_{\Ue(\frak{sl}(\infty))} L_{\frak b^0}(f_2)=I$$but $L_{\frak b^0}(f_1)\not\cong L_{\frak b^0}(f_2)$.

Assume now that $I$ has locally finite codimension. Then $I=I(Q)$ for an irreducible c.l.s. of finite type $Q$, and by Proposition~\ref{Lexpl} c) $M$ is isomorphic to $L_{\frak b^0}(f^0)$ for some ideal Borel subalgebra $\frak b^0$ and $\frak b^0$-dominant function $f^0$. Moreover, as $I$ is clearly fixed under the group $\tilde G:= \{g\in{\rm Aut}_\mathbb CV(\infty)\mid g^*(V(\infty)_*)=V(\infty)_*\}$ considered as a group of automorphisms of $\Ue(\frak{sl}(\infty))$, it follows that $M$ is invariant under $\tilde G$. Now Theorems~3.4 and~4.2 in~\cite{DPS} imply that $L_\frak b(f)$ is a simple tensor module.

It remains to show that a simple tensor $\frak{sl}(\infty)$-module $M$ is determined up to isomorphism by its annihilator $\Ann_{\Ue(\frak{sl}(\infty))}M$. If $M'$ is a simple $\frak{sl}(\infty)$-module with $\Ann_{\Ue(\frak{sl}(\infty))}M'=\Ann_{\Ue(\frak{sl}(\infty))} M=I$, then the fact that $I$ has locally finite codimension implies that $M'$ is integrable and that the c.l.s. of $M'$ coincides with the c.l.s. of $I$, i.e., $\cls(M)=\cls(M')$. Furthermore, a careful analysis (carried out in detail in A.~Sava's master's thesis~\cite{Sa}) shows that $M'$ is a highest weight $\frak{sl}(\infty)$-module with respect to the ideal Borel subalgebra given by the order (\apro{3}) of Section~\ref{SSslnot}, and that the highest weight of $M$ equals the highest weight of $M'$. This of course implies that $M'\cong M$.\end{proof}

\end{document}